\documentclass[12pt]{amsart}
\usepackage[margin=1in]{geometry}
\usepackage{amsmath,amsfonts,amsthm,amssymb,bbm}
\usepackage{graphicx,color,dsfont}
\usepackage{fourier}
\usepackage[pagewise]{lineno}

\begin{document}

\newtheorem{theorem}{Theorem}
\newtheorem{lemma}{Lemma}
\newtheorem{proposition}{Proposition}
\newtheorem{rmk}{Remark}
\newtheorem{example}{Example}
\newtheorem{exercise}{Exercise}
\newtheorem{definition}{Definition}
\newtheorem{corollary}{Corollary}
\newtheorem{notation}{Notation}
\newtheorem{claim}{Claim}

\newtheorem{dif}{Definition}

 \newtheorem{thm}{Theorem}[section]
 \newtheorem{cor}[thm]{Corollary}
 \newtheorem{lem}[thm]{Lemma}
 \newtheorem{prop}[thm]{Proposition}
 \theoremstyle{definition}
 \newtheorem{defn}[thm]{Definition}
 \theoremstyle{remark}
 \newtheorem{rem}[thm]{Remark}
 \newtheorem*{ex}{Example}
 \numberwithin{equation}{section}

\newcommand{\vertiii}[1]{{\left\vert\kern-0.25ex\left\vert\kern-0.25ex\left\vert #1
    \right\vert\kern-0.25ex\right\vert\kern-0.25ex\right\vert}}

\newcommand{\R}{{\mathbb R}}
\newcommand{\C}{{\mathbb C}}
\newcommand{\U}{{\mathcal U}}
\newcommand{\norm}[1]{\left\|#1\right\|}
\renewcommand{\(}{\left(}
\renewcommand{\)}{\right)}
\renewcommand{\[}{\left[}
\renewcommand{\]}{\right]}
\newcommand{\f}[2]{\frac{#1}{#2}}
\newcommand{\im}{i}
\newcommand{\cl}{{\mathcal L}}
\newcommand{\ck}{{\mathcal K}}

\newcommand{\al}{\alpha}
\newcommand{\be}{\beta}
\newcommand{\wh}[1]{\widehat{#1}}
\newcommand{\ga}{\gamma}
\newcommand{\Ga}{\Gamma}
\newcommand{\de}{\delta}
\newcommand{\ben}{\beta_n}
\newcommand{\De}{\Delta}
\newcommand{\ve}{\varepsilon}
\newcommand{\ze}{\zeta}
\newcommand{\Th}{\Theta}
\newcommand{\ka}{\kappa}
\newcommand{\la}{\lambda}
\newcommand{\laj}{\lambda_j}
\newcommand{\lak}{\lambda_k}
\newcommand{\La}{\Lambda}
\newcommand{\si}{\sigma}
\newcommand{\Si}{\Sigma}
\newcommand{\vp}{\varphi}
\newcommand{\om}{\omega}
\newcommand{\Om}{\Omega}
\newcommand{\ra}{\rightarrow}

\newcommand{\ro}{{\mathbf R}}
\newcommand{\rn}{{\mathbf R}^n}
\newcommand{\rd}{{\mathbf R}^d}
\newcommand{\rmm}{{\mathbf R}^m}
\newcommand{\rone}{\mathbf R}
\newcommand{\rtwo}{\mathbf R^2}
\newcommand{\rthree}{\mathbf R^3}
\newcommand{\rfour}{\mathbf R^4}
\newcommand{\ronen}{{\mathbf R}^{n+1}}
\newcommand{\ku}{\mathbf u}
\newcommand{\kw}{\mathbf w}
\newcommand{\kf}{\mathbf f}
\newcommand{\kz}{\mathbf z}

\newcommand{\N}{\mathbf N}

\newcommand{\tn}{\mathbf T^n}
\newcommand{\tone}{\mathbf T^1}
\newcommand{\ttwo}{\mathbf T^2}
\newcommand{\tthree}{\mathbf T^3}
\newcommand{\tfour}{\mathbf T^4}

\newcommand{\zn}{\mathbf Z^n}
\newcommand{\zp}{\mathbf Z^+}
\newcommand{\zone}{\mathbf Z^1}
\newcommand{\zz}{\mathbf Z}
\newcommand{\ztwo}{\mathbf Z^2}
\newcommand{\zthree}{\mathbf Z^3}
\newcommand{\zfour}{\mathbf Z^4}

\newcommand{\hn}{\mathbf H^n}
\newcommand{\hone}{\mathbf H^1}
\newcommand{\htwo}{\mathbf H^2}
\newcommand{\hthree}{\mathbf H^3}
\newcommand{\hfour}{\mathbf H^4}

\newcommand{\cone}{\mathbf C^1}
\newcommand{\ctwo}{\mathbf C^2}
\newcommand{\cthree}{\mathbf C^3}
\newcommand{\cfour}{\mathbf C^4}
\newcommand{\dpr}[2]{\langle #1,#2 \rangle}

\newcommand{\sn}{\mathbf S^{n-1}}
\newcommand{\sone}{\mathbf S^1}
\newcommand{\stwo}{\mathbf S^2}
\newcommand{\sthree}{\mathbf S^3}
\newcommand{\sfour}{\mathbf S^4}

\newcommand{\lp}{L^{p}}
\newcommand{\lppr}{L^{p'}}
\newcommand{\lqq}{L^{q}}
\newcommand{\lr}{L^{r}}
\newcommand{\echi}{(1-\chi(x/M))}
\newcommand{\chip}{\chi'(x/M)}

\newcommand{\wlp}{L^{p,\infty}}
\newcommand{\wlq}{L^{q,\infty}}
\newcommand{\wlr}{L^{r,\infty}}
\newcommand{\wlo}{L^{1,\infty}}

\newcommand{\lprn}{L^{p}(\rn)}
\newcommand{\lptn}{L^{p}(\tn)}
\newcommand{\lpzn}{L^{p}(\zn)}
\newcommand{\lpcn}{L^{p}(\cn)}
\newcommand{\lphn}{L^{p}(\cn)}

\newcommand{\lprone}{L^{p}(\rone)}
\newcommand{\lptone}{L^{p}(\tone)}
\newcommand{\lpzone}{L^{p}(\zone)}
\newcommand{\lpcone}{L^{p}(\cone)}
\newcommand{\lphone}{L^{p}(\hone)}

\newcommand{\lqrn}{L^{q}(\rn)}
\newcommand{\lqtn}{L^{q}(\tn)}
\newcommand{\lqzn}{L^{q}(\zn)}
\newcommand{\lqcn}{L^{q}(\cn)}
\newcommand{\lqhn}{L^{q}(\hn)}

\newcommand{\lo}{L^{1}}
\newcommand{\lt}{L^{2}}
\newcommand{\li}{L^{\infty}}
\newcommand{\beqn}{\begin{eqnarray*}}
\newcommand{\eeqn}{\end{eqnarray*}}
\newcommand{\pplus}{P_{Ker[\cl_+]^\perp}}

\newcommand{\co}{C^{1}}
\newcommand{\ci}{C^{\infty}}
\newcommand{\coi}{C_0^{\infty}}

\newcommand{\ca}{\mathcal A}
\newcommand{\cs}{\mathcal S}
\newcommand{\cm}{\mathcal M}
\newcommand{\cf}{\mathcal F}
\newcommand{\cb}{\mathcal B}
\newcommand{\ce}{\mathcal E}
\newcommand{\cd}{\mathcal D}
\newcommand{\cn}{\mathcal N}
\newcommand{\cz}{\mathcal Z}
\newcommand{\crr}{\mathbf R}
\newcommand{\cc}{\mathcal C}
\newcommand{\ch}{\mathcal H}
\newcommand{\cq}{\mathcal Q}
\newcommand{\cp}{\mathcal P}
\newcommand{\cx}{\mathcal X}
\newcommand{\eps}{\epsilon}

\newcommand{\pv}{\textup{p.v.}\,}
\newcommand{\loc}{\textup{loc}}
\newcommand{\intl}{\int\limits}
\newcommand{\iintl}{\iint\limits}
\newcommand{\dint}{\displaystyle\int}
\newcommand{\diint}{\displaystyle\iint}
\newcommand{\dintl}{\displaystyle\intl}
\newcommand{\diintl}{\displaystyle\iintl}
\newcommand{\liml}{\lim\limits}
\newcommand{\suml}{\sum\limits}
\newcommand{\ltwo}{L^{2}}
\newcommand{\supl}{\sup\limits}
\newcommand{\df}{\displaystyle\frac}
\newcommand{\p}{\partial}
\newcommand{\Ar}{\textup{Arg}}
\newcommand{\abssigk}{\widehat{|\si_k|}}
\newcommand{\ed}{(1-\p_x^2)^{-1}}
\newcommand{\tT}{\tilde{T}}
\newcommand{\tV}{\tilde{V}}
\newcommand{\wt}{\widetilde}
\newcommand{\Qvi}{Q_{\nu,i}}
\newcommand{\sjv}{a_{j,\nu}}
\newcommand{\sj}{a_j}
\newcommand{\pvs}{P_\nu^s}
\newcommand{\pva}{P_1^s}
\newcommand{\cjk}{c_{j,k}^{m,s}}
\newcommand{\Bjsnu}{B_{j-s,\nu}}
\newcommand{\Bjs}{B_{j-s}}
\newcommand{\Ly}{L_i^y}
\newcommand{\dd}[1]{\f{\partial}{\partial #1}}
\newcommand{\czz}{Calder\'on-Zygmund}
\newcommand{\chh}{\mathcal H}

\newcommand{\lbl}{\label}
\newcommand{\beq}{\begin{equation}}
\newcommand{\eeq}{\end{equation}}
\newcommand{\beqna}{\begin{eqnarray*}}
\newcommand{\eeqna}{\end{eqnarray*}}
\newcommand{\bp}{\begin{proof}}
\newcommand{\ep}{\end{proof}}
\newcommand{\bprop}{\begin{proposition}}
\newcommand{\eprop}{\end{proposition}}
\newcommand{\bt}{\begin{theorem}}
\newcommand{\et}{\end{theorem}}
\newcommand{\bex}{\begin{Example}}
\newcommand{\eex}{\end{Example}}
\newcommand{\bc}{\begin{corollary}}
\newcommand{\ec}{\end{corollary}}
\newcommand{\bcl}{\begin{claim}}
\newcommand{\ecl}{\end{claim}}
\newcommand{\bl}{\begin{lemma}}
\newcommand{\el}{\end{lemma}}
\newcommand{\dea}{(-\De)^\be}
\newcommand{\naa}{|\nabla|^\be}
\newcommand{\cj}{{\mathcal J}}

\title[Ground states for the Kawahara equation ]
{On   the normalized ground states for the Kawahara equation  and a fourth order NLS}

\author[Iurii  Posukhovskyi]{Iurii Posukhovskyi }
 \address{ Department of Mathematics,
University of Kansas,
1460 Jayhawk Boulevard,  Lawrence KS 66045--7523, USA}

\email{i.posukhovskyi@ku.edu}

\thanks{ Posukhovskyi is partially supported from a graduate fellowship by NSF-DMS under grant  \# 1614734.   Stefanov    is partially  supported by  NSF-DMS under grant  \# 1908626.}

\author[Atanas Stefanov]{\sc Atanas Stefanov}
\address{ Department of Mathematics,
University of Kansas,
1460 Jayhawk Boulevard,  Lawrence KS 66045--7523, USA}
\email{stefanov@ku.edu}



\date{\today}
 
\begin{abstract}

We consider the Kawahara model and two fourth order semi-linear Schr\"odinger equations in any spatial dimension. 
We construct the corresponding normalized ground states, which we rigorously show to be spectrally stable. 

 For the Kawahara model,  our results provide a significant extension in parameter space of  the current  rigorous results. In fact, our results  establish (modulo an additional technical assumption, which should be satisfied at least generically),  spectral stability for all normalized waves constructed therein - in all dimensions, for all acceptable values of the parameters. This,  combined with the results of \cite{BCSN},   provides orbital stability, for all normalized waves enjoying the   non-degeneracy property. The validity of the 
 non-degeneracy property for generic waves remains an intriguing open question. 
 
 At the same time, we verify and clarify recent numerical simulations  of the spectral stability of these solitons. For the fourth order NLS models, we improve upon recent results on spectral stability of very special, explicit solutions in the one dimensional case. Our multidimensional results for fourth order anisotropic NLS seem to be the first of its kind. Of particular interest is a new paradigm that we discover herein.  Namely,  all else being equal, the form of the second order derivatives (mixed second derivatives vs. pure Laplacian) has implications on the range of existence and stability of the normalized waves. 
\end{abstract}

\maketitle
\section{ Introduction}
We consider several dispersive models  in one and multiple space dimensions. Our main motivating  example will  be the (generalized) Kawahara equation, which is a fifth order generalized KdV equation, which allows for third order dispersion effects as well.  Namely, we set 
\begin{equation}
\label{8}
u_t+u_{xxxxx}+b u_{xxx} - (|u|^{p-1} u)_x=0, x\in\rone, t\geq 0, p>1
\end{equation}
 This is a model that appears in the study of plasma   and  capillary waves, where the third order dispersion is considered to be weak. In fact, 
 Kawahara studied the quadratic case\footnote{where the nonlinearity is in the form $(u^2)_x$, slightly different than ours}  \cite{Kaw} and he argued  that the inclusion of a fifth order derivative is necessary for capillary-gravity waves, for values of the Bond number close to the  critical one. Craig and Groves, \cite{CG} offered some further generalizations.  Kichenassamy and Olver, \cite{KO}  have studied the cases where explicit waves exist, see also Hunter-Scheurle, \cite{HS}  for existence of solitary waves.  

 Another model, which is important in the applications,  is  the non-linear  Schr\"odinger equation with fourth order dispersion. We  consider two versions of it, which will turn out to be qualitatively different, from a the  point of view of the stability of their standing waves. More precisely, 
  \begin{eqnarray}
  \label{10} 
& &   i u_t +\De^2 u +  \eps (\dpr{\vec{b}}{\nabla})^2 u - |u|^{p-1} u = 0, \ \ \ (t,x)\in \rone\times \rd,\\
& &   \label{NLS2}
   i u_t +\De^2 u + b \De  u - |u|^{p-1} u = 0, \ \ \ (t,x)\in \rone\times \rd,
  \end{eqnarray}
  where $d\geq 1$, $p>1, \eps=\pm 1$. These have  been much studied, both in the NLS as well as Klein-Gordon context, since the early 90's, see for example  \cite{A, ACN}. 
  
 For both models,  we will be interested in the existence of solitons, and the corresponding close to soliton dynamics, in particular spectral stability.  
 For the Kawahara, the relevant objects are traveling  waves, in the form $u(x,t)=\phi(x+\om t)$, where $\phi$ is dying off at infinity. These  satisfy profile equation of the form 
  \begin{equation}
  \label{15} 
  \phi'''' +b \phi'' + \om \phi- |\phi|^{p-1} \phi = 0. 
  \end{equation} 
 Similarly, standing wave  solutions in the form $u=e^{-i \om t} \phi, \om>0$, with  real-valued $\phi$ for the fourth order NLS \eqref{10} and \eqref{NLS2} 
 solve the elliptic profile equations 
  \begin{eqnarray}
  \label{20} 
& &   \De^2 \phi +\eps  (\dpr{\vec{b}}{\nabla})^2 \phi + \om \phi- |\phi|^{p-1} \phi = 0 \\
    \label{NLS22} 
& &   \De^2 \phi +b \De \phi+ \om  \phi -|\phi|^{p-1} \phi=0. 
  \end{eqnarray} 
 Constructing solutions to \eqref{15}, and more generally \eqref{20} and \eqref{NLS22},  is not straightforward task. In fact, it depends on the parameter $p$, the sign of the parameter $b$, as well as the dimension $d\geq 1$. Here, it is worth noting the works of  Albert, \cite{A} and Andrade-Cristofani-Natali, \cite{ACN} in which the authors have mostly studied the stability of some explicitly available solutions in one spatial dimension. 
 
 We proceed differently, by means of variational methods. More specifically, we  employ the  constrained minimization method, which  minimizes  total energy with respect to a fixed particle number, or $L^2$ mass. In addition to being the most physically relevant, the waves constructed this way (which we refer to henceforth as normalized waves) have good stability properties. 
 
 This brings us to the second important goal of the paper. Namely, we wish to examine the spectral stability of waves arising as solutions of  \eqref{15} and \eqref{20}. Our constructions will not yield explicit waves\footnote{although some do exist, for very specific values of the parameter $b$ and $d=1$, more on this below}. Thus, we need to decide about their stability, based on  their construction and properties.
 \subsection{Previous results}
  \subsubsection{The Kawahara model} 
   We would like to review the history of the problem for existence and stability of the traveling waves. We    concentrate  mostly on some recent results in the last twenty  years or so, which we feel  are most pertinent  to our results. We would like to emphasize  an important point, namely that since uniqueness results are generally lacking\footnote{both as minimizers of constrained variational problem and as solutions of the PDE}, it is hard to compare different results about waves obtained by different methods, as they may be different in shape and stability properties. 
   
    In \cite{IS}, \cite{K}, the authors have shown that certain waves of depression (i.e. $b<0$) are stable. In \cite{K}, the author establishes an important, Vakhitov-Kolokolov type criteria for certain waves, but it appears that it is hard to verify outside of a few explicit examples.  In \cite{BD}, Bridges and Derks, have studied a Kawahara-type  model, with more general nonlinearity.  They have employed the Evans function method to locate the point 
    spectrum (and hence the stability) of the corresponding linearizations. The results of their work are mostly computationally aided. 
    
    Levandosky, 
    \cite{L1} has studied the problem for existence of such waves via an energy - momentum type argument and concentration  compactness. Groves, \cite{G} has shown the existence    of multi-bump solitary waves for certain homogeneous nonlinearities.       
    Haragus-Lombardi-Scheel, \cite{HLS}  have considered spatially periodic solutions and solitary waves, which are asymptotic to them at infinity. They showed spectral stability for such small amplitude solutions.  We should also mention the work \cite{ACN}, in which the authors consider the orbital stability for  explicit {\it periodic} solutions of the Kawahara problem, subjected to a quadratic nonlinearity. 
    
    The paper of Angulo, \cite{Pava} gives some sufficient conditions for instability of such waves, both for the cases $b>0$ and $b<0$.  Levandosky, \cite{L} nicely summarizes the results in the literature\footnote{but he considers more general non-linearities, containing powers of derivatives as well}   and offers  rigorous analysis for stability/instability close to bifurcation points. Furthermore, his paper provides an useful, numerically aided,  classification of solitary waves of the Kawahara model, based on the type of non-linearity (i.e. the power $p$) and the parameters of the problem $b, \om$.  The exhaustive tables on p. 164, \cite{L} provided a good starting point for our investigation. We should mention that the waves considered  in \cite{L} are   produced as  the constrained minimizers of the following variational problem
  \begin{equation}
  \label{Lev}
  \left\{
  \begin{array}{l}
  J_\om[u]=\int_{\rd} |\De u(x)|^2 - b |\nabla u(x)|^2+ \om u^2(x) dx \to \min \\
  \int_{\rd} |u(x)|^{p+1} dx=1 
  \end{array}
  \right.
  \end{equation}
  {\it We take different approach below,  by constructing the normalized waves. These are the waves that  precisely minimize energy, when one constrains the $L^2$ norm, see Section \ref{sec:3.1}. }

  An important point we would like to make however  is that the  procedure outlined by \eqref{Lev} provides waves for a considerably wider range of $p$,   than the ones produced in Section \ref{sec:3.1}. Namely, the minimizers of \eqref{Lev} exist for $p\in (1, p_{\max})$, with 
   $
 p_{\max}(d)=\left\{ 
 \begin{array}{cc}
 \infty & d=1,2,3,4 \\
 1+ \f{8}{d-4} & d\geq 5
 \end{array}
 \right.
 $
whereas, the normalized waves constructed herein are only available for $p\in (1, 1+\f{8}{d})$. 
    
\subsubsection{Fourth order NLS model}      The fourth order Schr\"odinger equation was introduced in  \cite{KS2}, \cite{K1}, where it plays an important role in modeling the propagation of intense laser beams in a bulk medium with Kerr nonlinearity. Moreover, the equation was also used in nonlinear fiber optics and the theory of optical solitons in gyro tropic media. The problem for the existence and the stability of the waves arising in \eqref{20} has been the subject of investigations of a few recent works, the results of which we summarize below. 

For the case of $d=1, p=3$  (and in fact only for the special value of $\eps=-1, b=1$ and $\om=\f{4}{25}$), the elliptic problem  \eqref{15} (or equivalently \eqref{20}) 
was considered by Albert, \cite{A} in relation to soliton solutions to related approximate water wave models. The explicit soliton,  $\phi_0(x) = \sqrt{\f{3}{10}} sech^2\left(\f{x}{\sqrt{20}}\right)$,  was studied in detail in \cite{A}.  Important properties of the corresponding  linearized operators were established.  These properties allowed Natali and Pastor, \cite{NA} to establish the orbital stability of this wave, see also \cite{FSS} for alternative approach and extensions to Klein-Gordon solitons.   One of the central difficulties that the authors faced is that this solution is only available explicitly for an isolated value of\footnote{ which precludes one from differentiating with respect to the parameter $\om$ as is customary in these types of arguments}  $\om=\f{4}{25}$.  
Additionally, the problem for stability of the equation \eqref{10} in $d=1$, $\eps=-1, b=1$ and general $p$ were addressed in the works    \cite{K6} and \cite{KS1}. The numerically generated waves were shown    to exists for every $p>1$, but they are stable only  for $p\in (1,5)$. Further (mostly numerical) investigations 
regarding this model  are available in the papers \cite{KS2}, \cite{K1}. 

Finally, it is important to discuss the recent work \cite{BCSN}, as it has significant overlap with ours. In it, the authors have studied \eqref{NLS2} in great detail, including the stability of the  waves. They have constructed the waves in a similar manner, in fact the existence part of our Theorem \ref{theo:NLS2} is similar in nature\footnote{although more details on radial symmetry, the zero set and exponential decay   of the waves are derived as well}.  In addition, they   discuss some cases, in which they can show the important non-degeneracy property, that is  $Ker[\cl_+]=span[\nabla \phi]$.  This is rigorously verified  in two cases only: 
\begin{itemize}
		\item the one dimensional case, $d=1$, with $b<0$, $b^2>4\om$. 
	\item for any dimension $d\geq 2$,  but with $b<0$ and $|b|$ sufficiently  large,
\end{itemize} 
{\it Concerning stability of the waves, the authors of \cite{BCSN}  do not actually establish stability for any given example.}  On the other hand, they show that orbital stability holds, once one can verify non-degeneracy and the index condition $\dpr{\cl_+^{-1}  \phi}{\phi}<0$.  The concrete 
	details of these results are provided in  \cite{BCSN}, although this is a more general theorem, see for example Theorem 5.2.11, \cite{KP}.  The non-degeneracy was already discussed, while the verification of  $\dpr{\cl_+^{-1}  \phi}{\phi}<0$ is left as an open problem in \cite{BCSN}. This last condition however is essentially equivalent, modulo some easy to establish technical assumptions,  to the spectral stability, see Corollary \ref{stability} below. 

{\it  In this work, we actually do show $\dpr{\cl_+^{-1}  \phi_\la}{\phi_\la}\leq 0$  for all waves produced in Theorems \ref{theo:Kaw}, \ref{theo:NLS}, \ref{theo:NLS2}, thus answering the open problem  in \cite{BCSN}. With the exception  of  the case $\dpr{\cl_+^{-1}  \phi_\la}{\phi_\la}=0$ (which is  a non-degeneracy condition of sort, that we cannot rule out),  our results provide rigorously for spectral stability for all waves constructed therein - in all dimensions $d\geq 1$, 
	for all allowed values of $b: d=1, b\in \rone$ and $d\geq 2, b<0$. This, in combination with the results of \cite{BCSN},   shows  orbital stability, for all normalized waves enjoying the  non-degeneracy property of the wave as well as the property $\dpr{\cl_+^{-1} \vp_\la}{\vp_\la}\neq 0$.}

 \subsection{Main results: Kawahara waves} 
 It is easy to informally summarize our results - all  normalized waves, whenever they exist, turn out to be spectrally stable. This is an interesting paradigm, which is currently under investigation in a variety of models. Our hope is that the approach here will shed further light on this interesting phenomena in a much more general setting. 
 As we have alluded to above, our focus will be the Kawahara problem, \eqref{8}, 
 for both positive and negative values of $b$.   
\subsubsection{Kawahara waves: Existence} 
 In order to construct solutions to the elliptic problem \eqref{15}, we shall work with the following variational problem 
 \begin{equation}
 \label{70}
 \left\{ 
 \begin{array}{l}
 I[\phi]=\frac{1}{2}\int_{\rone}[|\phi''(x)|^2-b|\phi'(x|^{2}]dx-\frac{1}{p+1}\int_{\rone}|\phi(x)|^{p+1}dx\to \min \\
 \int_\rone \phi^2(x) dx=\la,
 \end{array}
 \right.
 \end{equation}
 where one could take $\phi$ in the Schwartz class,  in order to make $I[\phi]$ meaningful. 
 Introduce the scalar function 
 $$
 m_b(\lambda)=\inf_{\phi\in H^2(\rone),\left\lVert \phi \right\rVert_2^2=\lambda}I[\phi],
 $$
which  plays a prominent form in the subsequent arguments. Let us emphasize that it is not {\it a priori} clear whether the problem \eqref{70} is well-posed (i.e. $m_b(\la)>-\infty$) for all $\la$. We have the following existence result.

 \begin{theorem}(Existence of the normalized Kawahara traveling waves) \\ 
 \label{theo:Kaw}
 Let $p \in (1,9),\la>0, b\in \rone$ satisfy one of the following 
 \begin{enumerate}
 \item $1<p<5, \la>0$
 \item For  $5\leq p<9$ and all sufficiently large\footnote{Here, for all given $p\in [5,9)$, for both $b>0, b<0$, 
  there is a specific value$\la_{b,p}$ and we assume that  $\la>\la_{b,p}$}   $\la$ 
 \end{enumerate}
 Then, the constrained minimization problem \eqref{70} has a solution, $\phi_\la\in H^4(\rone): \|\phi\|_{L^2}^2=\la$ and  $\om=\om(b,\la, \phi)$. 
Moreover, $\phi_\la$ satisfies the Euler-Lagrange equation \eqref{15}   in a classical  sense. 
 We call such solutions $\phi_\la$ normalized waves. 
 \end{theorem}{\bf Remark:} 
   The Lagrange multiplier $\om$ may depend on the normalized wave $\phi$. In particular, we can  not rule out the existence of two constrained minimizers of \eqref{70}, $\phi_\la, \tilde{\phi}_\la$, with $\om(\la, \phi_\la)\neq \om(\la, \tilde{\phi}_\la)$. This is of course related to the uniqueness problem for the minimizers of \eqref{70}  (and it should be a much simpler one), but it is open at the moment. 
 
  \subsubsection{Kawahara waves: stability} 
  
  We now discuss our  results concerning the stability of the   waves produced in Theorem \ref{theo:Kaw} - we employ the standard definition of spectral stability, see Definition \ref{defi:lin} in Section \ref{sec:2.1} below.  
  Before we give the formal statements, we need to state an important  property of the waves $\phi$ constructed in Theorem \ref{theo:Kaw}. Namely, upon introducing the self-adjoint  linearized operator 
  	$$
  	\cl_+=\p_x^4+b \p_x^2 +\om_{b,\la}- p |\phi_\la|^{p-1},
  	$$
  {\it we say that $\phi_\la$ is weakly non-degenerate, if $\phi_\la\perp Ker[\cl_+]$}. In  particular, $\cl_+^{-1} \phi_\la$ is well-defined. 	   
  \begin{theorem}
  	\label{theo:Kawstab} 
  	Let $\la>0$ and $p$ satisfy the requirements of Theorem \ref{theo:Kaw}, and $\phi_\la$ is any minimizer constructed therein. Then,  $\phi_\la$ is weakly non-degenerate. If in addition, the condition  $\dpr{\cl_+^{-1} \phi_\la}{\phi_\la}\neq 0$ is satisfied, then the wave $\phi_\la$ is spectrally stable, as a solution to the Kawahara problem \eqref{8}, in the sense of  Definition \ref{defi:lin} below. 
  \end{theorem}
  {\bf Remarks:} 
  \begin{itemize}
   \item The condition $\dpr{\cl_+^{-1} \phi_\la}{\phi_\la}\neq 0$ appears frequently as a non-degeneracy condition in the literature, \cite{KP}. It is worth noting that such a condition has a clear physical spectral meaning, namely that the  eigenvalue at zero for $\p_x \cl_+$, generated by the translational  invariance,  has an associated  Jordan cell of order exactly two. Physically, such an eigenvalue is expected to be of algebraic multiplicity exactly two and geometric multiplicity one, as this is the only invariance in the system, so this must hold generically. We do not have a rigorous proof of this fact at the moment.
  	\item The results of Theorem \ref{theo:Kawstab} present rigorous sufficient conditions for stability of traveling waves in much wider range than previously available.  In fact, our results confirm\footnote{With the usual caveat, that since there is no uniqueness, it is possible that the waves considered in \cite{L} are different than ours!}  the available numerical simulations by Levandosky,  \cite{L}. For example, it is quite obvious that the bifurcation point   is at\footnote{corresponds to the case $p=6$ in  the notations of \cite{L}} $p=5$.  More precisely, for powers $p<5$ all waves are stable\footnote{except at $p=4$ ($p=5$ in the notations of \cite{L}) - for a small region in the 
  		parameter space, an instability is observed numerically.  This must be a  fluke of the computations in \cite{L}, because as we see from Theorem \ref{theo:Kaw}, the stable region is up to $p< 5$}, while for $p> 5$, some unstable waves start to appear (which are of course not  normalized). For $p\geq 9$, Levandosky observed a   very small set of stable waves, again none of them normalized, but rather generated as minimizers of \eqref{Lev}. 
  	
  	\item The Cauchy problem for the particular version of the Kawahara problem \eqref{8} considered herein,  has not been studied methodically, to the best of our knowledge. Based on the results of the standard NLS though, one might conjecture that the problem is globally well-posed for all values $1<p<9$. An important related issue is the conservation of Hamiltonian, momentum and $L^2$ mass along the evolution of solutions emanating from sufficiently nice data. 
  	\item In the presence of satisfactory well-posedness theory, as outlined above, nonlinear (or strong orbital) stability of the wave $\phi(x+\om t)$ follows  from our arguments, once one can establish that the linearized operator 
  	$\cl_+$ has one dimensional kernel, namely $Ker[\cl_+]=span [\phi']$. This is in essence standard, but it does not follow directly within  the Grillakis-Shatah-Strauss  formalism, \cite{GSS}, since this approach would require the smoothness of the mapping $\la\to \phi_\la$, which is currently unknown.    In particular, we refer to a method pioneered by T. B. Benjamin in \cite{Ben}, for the stability of the KdV waves, which has since been refined and improved by other authors.  On the other hand, we refer to the arguments for the NLS case to \cite{BCSN}.  	 
  	
  	\item The non-degeneracy  $Ker[\cl_+]=span [\phi']$ appears to be a hard problem in the theory. An easier version would be to establish such a non-degeneracy of the kernel, if $\phi$ is a minimizer of \eqref{70}.  A harder problem would be to do so, knowing that $\phi$ is just a solution to the PDE \eqref{15}. In both cases, the non-degeneracy is directly relevant  to the uniqueness of the ground state, which is even harder open  problem in the area.  See \cite{FL} for discussion about these and related issues. 
  \end{itemize}  
    \subsection{Main results: fourth order NLS waves} 
   We start with the existence result for the models. 
   \subsubsection{Existence of normalized waves for  fourth order NLS models} 
 Before we state the results for the  fourth order NLS models, we need to make an obvious reduction of the equation \eqref{10}. Namely, picking a  matrix $A\in SU(n)$, so that $\vec{b}=|\vec{b}| A \vec{e_1}$, we can clearly reduce matters (both the existence of the solutions of the profile equation \eqref{20} and its stability analysis), by the transformation $\hat{u}(\xi)\to \hat{u}(A^* \xi)$,  to the following problem:  
 \begin{equation}
 \label{712}
 i u_t+\De^2 u+\eps |b|^2 \p_{x_1}^2 u - |u|^{p-1} u=0
 \end{equation}
 and its associated elliptic profile equation 
 \begin{equation}
 \label{715}
 \De^2 \phi +\eps |b|^2 \p_{x_1}^2 \phi+\om \phi  - |\phi|^{p-1} \phi=0. 
 \end{equation}
 That is, the existence of solutions to \eqref{715} is equivalent to the existence of solutions to \eqref{20} (under the appropriate transformation) and their stability is equivalent to the stability of their  counterparts. Thus, it suffices to discuss the fourth order NLS problem \eqref{712}, with its solitons satisfying  \eqref{715}. 
 Our variational setup in the anisotropic case is as follows 
 \begin{equation}
 \label{700}
 \left\{ 
 \begin{array}{l}
 I[\phi]=\frac{1}{2}\int_{\rd}[|\De \phi (x)|^2-\eps |\vec{b}|^2 |\p_{x_1}\phi(x)|^2]dx-\frac{1}{p+1}\int_{\rd}|\phi(x)|^{p+1}dx\to \min \\
 \int_{\rd} \phi^2(x) dx=\la,
 \end{array}
 \right.
 \end{equation}

  \begin{theorem}(Stability of the normalized   waves for the fourth order  NLS: mixed derivatives) \\ 
 \label{theo:NLS}
 Let $d\geq 1, \eps=- 1$. Let  $p  \in (1, 1+\f{8}{d})$, $\la>0$ and  
   \begin{enumerate}
 \item $1<p<1+\f{8}{d+1}, \la>0$
 \item If $1+\f{8}{d+1} \leq p< 1+\f{8}{d}$, assume a sufficiently large    $\la$.
 \end{enumerate}
Then, there exists $\phi\in H^4(\rd)\cap L^{p+1}(\rd)$ satisfying \eqref{715}, with an appropriate $\om=\om(\la, \phi)$.  

  The wave $\phi_\la$ is constructed as constrained minimizer of \eqref{700}, with 
  $\|\phi_\la\|_{L^2}^2=\la$.  Assuming in addition the condition $\dpr{\cl_+^{-1} \phi_\la}{\phi_\la}\neq 0$, then $e^{ - i \om_\la t} \phi_\la(x)$ is a spectrally stable solution of \eqref{712},  in the sense of  Definition \ref{defi:lin} below. 
  \end{theorem}
  {\bf Remark:} The case $\eps=1$, in the higher dimensions $d\geq 2$, while undoubtedly interesting in the applications,  is much more subtle, and it cannot be analyzed with the methods of this paper.  We will address some aspects of it in a forthcoming publication \cite{KST}.

Despite the obvious similarities with \eqref{20}, the fourth order NLS with pure Laplacian, \eqref{NLS2} and its associated profile equation \eqref{NLS22},   turn out  quite different -  even at the level of the existence of the waves and their stability. We introduce the relevant variational problem
\begin{equation}
\label{701}
\left\{ 
\begin{array}{l}
I[\phi]=\frac{1}{2}\int_{\rd}[|\De \phi (x)|^2-
b  |\nabla \phi(x)|^2]dx-\frac{1}{p+1}\int_{\rd}|\phi(x)|^{p+1}dx\to \min \\
\int_{\rd} \phi^2(x) dx=\la,
\end{array}
\right.
\end{equation}
 \begin{theorem}(Stability of the normalized  waves for the fourth order  NLS: pure Laplacian case) \\ 
 \label{theo:NLS2}
 Let $d\geq 1$, $b<0$. Let  $p  \in (1, 1+\f{8}{d})$,  $\la>0$ and  
   \begin{enumerate}
 \item $1<p<1+\f{4}{d}, \la>0$
 \item If $1+\f{4}{d} \leq p< 1+\f{8}{d}$, assume a sufficiently large    $\la$.
 \end{enumerate}
Then, there exists a normalized wave $\phi_\la\in H^4(\rd)\cap L^{p+1}(\rd):\|\phi_\la\|^2=\la$,  satisfying \eqref{NLS22}, with an appropriate $\om=\om(\la, \phi)$. 
 The soliton $e^{ - i \om_\la t} \phi_\la(x)$ is a spectrally stable solution of \eqref{NLS2}, under  the additional condition $\dpr{\cl_+^{-1} \phi_\la}{\phi_\la}\neq 0$,  in the sense of  Definition \ref{defi:lin}. 
  \end{theorem}
 {\bf Remarks:}  
 \begin{itemize}
 	
  \item  The results extend the  stability results of Albert, \cite{A} for the one dimensional cubic case $p=3$.
  
  \item The results here also extend the NLS related  results of \cite{FSS} (namely, stability for $p<1+\f{8}{d}$ and  instability otherwise), which apply to the case $b=0$. 
 \item Both results, Theorem \ref{theo:NLS} and \ref{theo:NLS2} of course coincide for $d=1$, but are different for $d\geq 2$. We do not have a good physical explanation as to  why the range of existence and stability of standing waves for the models \eqref{712} vis a vis \eqref{NLS2} differ. In particular, the mixed derivative model, \eqref{712} seems to support all stable normalized waves in the   wider range $p\in (1, 1+\f{8}{d+1}), \la>0$, compared to $p\in (1, 1+\f{4}{d})$ for  \eqref{NLS2}. This topic clearly merits further investigations. 
 \item The cases $b>0, d\geq 2$ will be   analyzed in a forthcoming publication, \cite{KST}. 
 \end{itemize}
  The rest of the paper is organized as follows. In Section \ref{prelim},  we show that distributional solutions of the elliptic problems  are in fact strong solutions. We also set up the relevant eigenvalue problems, and in regards to that, we review the relevant instability index counting theories and some useful corollaries. Finally, we present the Pohozaev identities, which imply some necessary conditions for the existence of the waves.    We also note that better necessary conditions (which are closer to what we conjecture are the optimal ones) are possible,  under a natural  spectral condition.  
  In Section \ref{sec:3}, we develop the existence theory in the one dimensional problem - this already contains all the difficulties, that one encounters in the higher dimensional situation as well. In particular, we discuss the well-posedness of the constrained minimization problem, the compensated compactness step, as well as the derivation of the Euler-Lagrange equation and various spectral properties of the linearized operators, which are useful in the sequel. 
  In Section \ref{sec:4}, we indicate the main steps in the variational construction for the waves in the higher dimensional case. In Section \ref{sec:5}, we provide a general framework for spectral stability, based on the index counting formula, which is easily applicable in our setting.  
  \section{Preliminaries} 
 \label{prelim}
 We first introduce some notations and standard inequalities.  We will frequently use the notation $f \lesssim g $, when $f,g$ are positive quantities/functions and there is a constant $C$, independent on the parameters so that $f\leq C g$. 
 \subsection{Function spaces and GNS inequalities}  The $L^p, 1\leq p<\infty$ spaces are defined via 
 $$
 \|f\|_{L^p}=\left(\int |f(x)|^p dx\right)^{1/p}, 
 $$
 For integer $k$,  the classical Sobolev spaces $W^{k.p}, 1\leq p<\infty$ are taken to be the closure of Schwartz functions in the norm 
 $\|f\|_{W^{k,p}}=\|f\|_{L^p}+\sum_{|\al|=k} \|\p^{\al} f\|_{L^p}$. 
 
 Next, we need some Fourier analysis basics. 
Fourier transform and its inverse are defined via 
$$
\hat{f}(\xi)=\int_{\rd} f(x) e^{-2\pi i x\cdot \xi} dx; \ \ f(x)=\int_{\rd} \hat{f}(\xi) e^{2\pi i x\cdot \xi} d\xi
$$
 Recall the sharp Sobolev inequality $\|f\|_{L^q(\rd)}\leq C_{s,p} \|f\|_{W^{s,p}(\rd)}$, where $1<p<q<\infty$ and  $s=n\left(\f{1}{p}-\f{1}{q}\right)$. Note that for non-integer values of $s$, the norm on the right-hand side is defined via 
 $$
 \|f\|_{W^{s,p}}:=\|(1-\De)^{s/2} f\|_{L^p},
 $$
 where  $\widehat{(1-\De)^{a} g}(\xi)=(1+4\pi^2 |\xi|^2)^a \hat{g}(\xi)$. 
 
 In addition, we shall make use of the Gagliardo-Nirenberg-Sobolev (GNS) inequality, which combines the Sobolev estimate with the well-known log-convexity of the  complex interpolation functor 
 $\|f\|_{[X_0, X_1]_\theta}\leq \|f\|_{X_0}^{1-\theta} \|f\|_{X_1}^{\theta}$.  For example, the following estimate proves useful in the sequel
 \begin{equation}
 \label{71} 
 \|u\|_{L^q(\rd)}\leq C_{q,d} \|\De u\|_{L^2}^{\f{d}{2}(\f{1}{2}-\f{1}{q})}  \|u\|_{L^2}^{1-\f{d}{2}(\f{1}{2}-\f{1}{q})},
 \end{equation}
 whenever $q\in (2,\infty)$, for $d=1,2,3,4$ and $2<q<\f{2d}{d-4}, d\geq 5$. 
 
 We record the formula for the Green function of $(-\De+1)^{-1}$, that is $\hat{Q}(\xi)=(1+4\pi ^2 |\xi|^2)^{-1}$ (see  \cite{G},  p. 418)
 \begin{equation}
 \label{g}
 Q(x)= (2\sqrt{\pi})^{-n} \int_0^\infty e^{-(t+\f{|x|^2}{4t})} \f{dt}{t^{n/2}}.
 \end{equation} 
 Note that $Q>0$, radial and radially decreasing.  Also, 
 $\|Q\|_{L^1(\rn)}=\int_{\rn} Q(x) dx=\hat{Q}(0)=1$, but note that $Q(0)=+\infty$ for $n\geq 2$.  In fact, 
 there are the following classical estimates for it,  p. 418, \cite{G},
 \begin{eqnarray}
 \label{loukas:10} 
 	|Q(x)|\leq   C e^{-|x|},\ \  |x|>1 \\
 	\label{loukas:20} 
 	Q(x)\sim \left\{ 
 	\begin{array}{cc}
 	|x|^{2-n}+O(1) & n\geq 3 \\
 	\ln(\f{1}{|x|})+O(1) & n=2
 	\end{array}
 	\right.\ \ |x|<1.
 \end{eqnarray}
In particular, $Q\in L^q(\rn)$, whenever $q<\f{n}{n-2}$ (or $q<\infty$, when $n=2$).

 \subsection{Distributional  vs strong solutions of the Euler-Lagrange equation} 
 \begin{definition}
 	\label{defi:12}
 	We say that $g\in H^2(\rd)\cap L^{p+1}(\rd)$ is a distributional solution of the equation
 	\begin{equation}
 	\label{515}
 	\De^2 g + b \De  g +\om g- |g|^{p-1} g=0, x\in \rd
 	\end{equation}
 	if   the following relation holds for every $h\in H^2(\rd)\cap L^\infty(\rd)$: 
 	$$
 	\dpr{\De g}{\De h}+  \dpr{b \De g+\om g}{h} - \dpr{g|^{p-1} g}{h}=0.
 	$$
 \end{definition}
 \begin{proposition}
 	\label{prop:nm}
 	Let $p\in (1, 1+\f{8}{d})$ and $b,\om$ be so that $b^2-4\om<0$ or $b^2-4\om>0, \om>0, b<0$. Then, any  weak solution $g$ of \eqref{515} is in fact $g\in H^4(\rd)\cap L^\infty(\rd)\cap L^{1+\eps}(\rd)$ for any $\eps>0$.  In particular,  the weak solutions of \eqref{515}   in fact satisfy \eqref{515} as $L^2$ functions.
 \end{proposition}
 \begin{proof}
  Note that by the restrictions on $b, \om$,  we have that the operator $(\De^2+b \De +\om)$ is invertible on $L^2(\rd)$. Let $\tilde{g}:=(\De^2+b \De +\om)^{-1} [|g|^{p-1} g]$. From Sobolev embedding, we easily get that $\tilde{g}\in H^{\al}(\rone), \al<4-\f{d(p-1)}{2(p+1)}$, since 
  $$
  \|\tilde{g}\|_{H^\al(\rd)}\leq \| |g|^{p-1} g\|_{H^{4-\al}(\rd)}\leq C \||g|^{p-1} g\|_{L^{\f{p+1}{p}}}  \leq C \|g\|_{L^{p+1}}^p.
  $$
  In addition, for every test function $h$, we have 
  $$
	\dpr{\De \tilde{g}}{\De h}+  \dpr{b \De \tilde{g}+\om \tilde{g}}{h} =  \dpr{|g|^{p-1} g}{h}= 	\dpr{\De g}{\De h}+  \dpr{b \De g+\om g}{h}.
  $$
  It follows that $g=\tilde{g}$ in the sense of distributions, whence $g\in H^\al(\rd)$. 
  We will show that $g\in L^\infty(\rd)$. Denote 
  $
  q_0=\sup\{q: g\in L^q(\rd)   \}.
  $
  Clearly, $q_0\geq p+1$, by assumption. We will show first that $q_0=\infty$. Assume not.  
  By Sobolev embedding, we have 
  $$
  \|g\|_{L^q(\rd)}= \|\tilde{g}\|_{L^q(\rd)}\leq C \||g|^{p-1} g\|_{L^{\f{p+1}{p}}} \leq C \|g\|_{L^{p+1}}^p<\infty 
  $$
  as long as $\f{1}{q}>\f{p}{p+1}-\f{4}{d}$. In particular, we can take $q$ as close to $\infty$ (and hence $q_0=\infty$), if $d\leq 4$. So, assume $d\geq 5$. It follows that $\f{1}{q_0}\leq \f{p}{p+1}-\f{4}{d}$. 
  
  Take any $q_0<q<\infty$. We have, by Sobolev embedding 
\begin{equation}
\label{contr} 
 \|\tilde{g}\|_{L^q(\rd)}\leq C \||g|^{p-1} g\|_{L^{r}}\leq C \|g\|_{L^{r p}}^p, 
\end{equation}
so long as $d(\f{1}{r}-\f{1}{q})\leq 4$ or $\f{1}{r}\leq \f{4}{d}+\f{1}{q}$. If $\f{4}{d}+\f{1}{q}<1$, we take $r: \f{1}{r}=\f{4}{d}+\f{1}{q}$, whereas, if we have $ \f{4}{d}+\f{1}{q}\geq 1$, we can take $r=\f{p+1}{p}$ and we have a contradiction right away, since the left-hand side of \eqref{contr} is   unbounded (by the definition of $q_0$),   while the right-hand  is bounded. For the remainder, take $r: \f{1}{r}=\f{4}{d}+\f{1}{q}$.

  Clearly, if $rp<q_0$, this would be a contradiction, because the left-hand side is supposed to be unbounded (by the definition of $q_0$), 
  while the right-hand side clearly is. We claim that this is the case, under our restrictions for $p\in (1, 1+\f{8}{d})$. We have 
  \begin{eqnarray*}
\f{1}{r}-\f{p}{q_0}= \f{4}{d}+\f{1}{q}-\f{p}{q_0}=\f{4}{d}-\f{p-1}{q_0}+o(q-q_0)
  \end{eqnarray*}
  So, if we show that $\f{4}{d}> \f{p-1}{q_0}$, we will have achieved the contradiction, as we can take $q$ very close to $q_0$.  Indeed, by the inequality for $\f{1}{q_0}$, we have 
  $
  \f{p-1}{q_0}\leq (p-1)\left(\f{p}{p+1}-\f{4}{d}\right)
$
Resolving the inequality 
$$
(p-1)\left(\f{p}{p+1}-\f{4}{d}\right)<\f{4}{d},
$$  
  leads to the solution $1<p<1+\f{8}{d-4}$, which of course contains the set $(1, 1+\f{8}{d})$, so it is true for all $p$ in the set that we are interested in. We have reached a contradiction, with $q_0<\infty$. 
  
  Thus, $q_0=\infty$. This does not mean yet that $g\in L^\infty(\rd)$, but this follows easily by Sobolev embedding, once we know that $g\in \cap_{2\leq q<\infty} L^q(\rd)$. Furthermore, we see that the same type of arguments imply $g\in H^5(\rd)$ and  that for every $p<\infty$ and for every $\eps>0$, 
  $g\in W^{4-\eps, p}(\rd)$.

  For our next step, we shall need a representation of the Green's function of the operator 
  $(\De^2+b \De +\om)^{-1}$ as follows. We have 
  \begin{eqnarray*}
 (\De^2+b \De +\om)^{-1} &=& (-\De+\f{-b+\sqrt{b^2-4\om}}{2})^{-1}(-\De+\f{-b-\sqrt{b^2-4\om}}{2})^{-1}= \\
 &=& (b^2-4\om)^{-1/2} [  (-\De+\f{-b-\sqrt{b^2-4\om}}{2})^{-1} - (-\De+\f{-b+\sqrt{b^2-4\om}}{2})^{-1}].
  \end{eqnarray*}
 In the case $b^2-4\om>0, \om>0, b<0$, both $\f{-b\pm\sqrt{b^2-4\om}}{2}$ are positive numbers, so clearly the corresponding Greens function $G$  has decay $e^{-\sqrt{\f{-b-\sqrt{b^2-4\om}}{2}}|x|}$, according to \eqref{loukas:10}.  
 
 As far as the case $b^2-4\om<0$ is concerned, it is not hard to see, in the same way,  that the Green's function $G$  has decay rate $e^{-k_\om|x|}$, where  
 $$
 k_\om:= \left\{ 
 \begin{array}{cc}
 \f{\sqrt{2\sqrt{\om}+b}}{2} & b<0 \\
  \f{\sqrt{2\sqrt{\om}-b}}{2} & b>0 
 \end{array} 
 \right. 
 $$
  In both cases,   the Green's function enjoys exponential rate of decay.  
  
  For $p\geq 2$, we can actually conclude that $g\in L^1(\rd)$ since by the 
  Hardy-Littlewood-Sobolev inequality 
  $$
  \|\tilde{g}\|_{L^1(\rd)} \leq \|G\|_{L^1(\rd)} \||g|^{p-1} g\|_{L^1(\rd)}\leq C\|g\|_{L^p(\rd)}^p<\infty,
  $$
  as $g\in L^2\cap L^\infty$, in particular $g\in L^p(\rd)$. For $p<2$, denote 
  $q_0=\inf\{q: g\in L^q(\rd)\}$. Our claim is that $q_0=1$. Assume for a contradiction that $q_0>1$. We will show that for every $q>q_0$, we have that $g\in L^{\f{q}{p}}(\rd)$, which would be a contradiction with $q_0>1$. Indeed, by Hardy-Littlewood-Sobolev 
  $$
  \|\tilde{g}\|_{L^{\f{q}{p}}(\rd)}\leq \|G\|_{L^1(\rd)} \||g|^{p-1} g\|_{L^{\f{q}{p}}(\rd)}\leq C\|G\|_{L^1} \|g\|_{L^q(\rd)}^p.
  $$
  This establishes the contradiction with $q_0>1$, hence $g\in \cap_{1<q} L^q(\rd)$.

 \end{proof}

 \subsection{Linearized problems and spectral stability}
  \label{sec:2.1}

 We next discuss the linearized problems and the stability of the waves. For solutions $\phi$ of \eqref{15}, we introduce the traveling wave ansatz, $u(t,x)=
 \phi(x+ \om t)+ v (t, x+t \om)$. Plugging this back in \eqref{8} and ignoring all terms $O(v^2)$, we obtain the following linearized problem 
 \begin{equation}
 \label{lin1}
 v_t+\p_x[\p_x^4+b \p_x^2+\om - p|\phi|^{p-1}]v=0.
 \end{equation}
 Denoting $\cl_+:=\p_x^4+b \p_x^2+\om - p|\phi|^{p-1}$, the associated eigenvalue problem is obtained by setting  $v(t,x) \to e^{-\mu t} z(x)$ in \eqref{lin1}, which results in 
  \begin{equation}
 \label{lin2}
 \p_x\cl_+ z=\mu z
 \end{equation}
 
We proceed similarly with the linearization of the NLS problem \eqref{10}. Consider  solutions $\phi$ of \eqref{715} and then 
perturbations of the  solution $u(t,x)=e^{ - i \om t} \phi$ of \eqref{712}  in the form 
 $u=e^{-i \om t}[\phi+z_1+i z_2]$. Plugging this ansatz into \eqref{10}, retaining only the linear in $z$ terms and taking real and imaginary parts leads us to the system 
 \begin{equation}
 \label{50}
 \p_t \left(\begin{array}{c} 
 z_1 \\ z_2 
 \end{array}\right) = \left(\begin{array}{cc} 
 0 & -1 \\ 1 & 0
 \end{array}\right) \left(\begin{array}{cc} 
 \De^2+ \eps |\vec{b}|^2 \p_{x_1}^2+\om - p |\phi|^{p-1}  & 0 \\ 0 &  \De^2+ \eps |\vec{b}|^2 \p_{x_1}^2 +\om -  |\phi|^{p-1}
 \end{array}\right)\left(\begin{array}{c} 
 z_1 \\ z_2 
 \end{array}\right)
 \end{equation}
 Thus, we introduce the scalar self-adjoint operators $\cl_\pm$ (note $\cl_+<\cl_-$) 
$$
 \left\{
 \begin{array}{ll}
 \cl_+ &= \De^2+ \eps |\vec{b}|^2 \p_{x_1}^2 + \om - p |\phi|^{p-1},  \\
 \cl_- &= \De^2+ \eps |\vec{b}|^2 \p_{x_1}^2 +\om -  |\phi|^{p-1}
 \end{array}
 \right. 
$$
 so that the eigenvalue problem associated with \eqref{50} and the assignment 
 $\vec{z}\to e^{\mu t} \vec{z}$,   takes the form 
 \begin{equation}
 \label{60}
 \cj \cl \vec{z}=\mu \vec{z}.
 \end{equation}
 where 
 $$
 \cj:=\left(\begin{array}{cc} 
 0 & -1 \\ 1 & 0
 \end{array}\right), \cl:=  \left(\begin{array}{ll} 
 \cl_+   & 0 \\ 0 &  \cl_-
 \end{array}\right). 
 $$
 Finally, for solutions $\phi$ of \eqref{NLS22}, the linearized problem appears in the form 
 \begin{equation}
 \label{52}
 \p_t \left(\begin{array}{c} 
 z_1 \\ z_2 
 \end{array}\right) = \left(\begin{array}{cc} 
 0 & -1 \\ 1 & 0
 \end{array}\right) \left(\begin{array}{cc} 
 \De^2+ b \De +\om - p |\phi|^{p-1}  & 0 \\ 0 &  \De^2+b\De+\om -  |\phi|^{p-1}
 \end{array}\right)\left(\begin{array}{c} 
 z_1 \\ z_2 
 \end{array}\right)
 \end{equation}
 This is again in the form \eqref{60}, with 
 $$
 \left\{
 \begin{array}{ll}
 \cl_+ &= \De^2+b\De + \om - p |\phi|^{p-1},  \\
 \cl_- &= \De^2+ b \De +\om -  |\phi|^{p-1}. 
 \end{array}
 \right. 
$$ 
We are now ready to give the definition of spectral  stability. Note that the essential spectrum   is, by Weyl's theorem,  is the range of the function  $\xi\in\rd \to |\xi|^4-b|\xi|^2+\om$. Clearly,  this  is the interval $[\om-\f{b^2}{4}, \infty)$, when $b>0$ and $[\om, \infty)$, when $b<0$.

\begin{definition}
\label{defi:lin}
The Kawahara waves are spectrally stable, provided the eigenvalue problem \eqref{lin2} does not have non-trivial solutions\footnote{Note that by the Hamiltonian symmetry of the problem $\mu\to -\mu$, the existence of eigenvalues $\mu: \Re \mu<0$ is equivalent to the existence of $\mu: \Re \mu>0$} 
$(\mu, z): \Re \mu>0, z\in H^5(\rone)$. 

The waves $\phi$ are spectrally stable, if the eigenvalue problems \eqref{60} (\eqref{52} respectively) do not have non-trivial solutions $(\mu, \vec{z}): \Re \mu>0, \vec{z}\in 
H^4(\rd)\times H^4(\rd)$. 
\end{definition}
  \subsection{Stability of linearized systems and index counting theories} 
 \label{sec:2.33}
  We need a quick introduction of   the instability index count theory, as developed in \cite{KKS}, \cite{KKS2}, \cite{Pel} (see also the book \cite{KP})  and more recently in \cite{KS}, \cite{LZ}. We will  only consider appropriate representative corollaries, which serve our purposes. For the purposes of this paper, we will follow closely the approach and the notations in \cite{LZ}. 
 	 To that end, we consider an  eigenvalue problem in the form\footnote{Before we embark on further details, let us once again emphasize that the examples that we will be interested in herein will be either in the form \eqref{lin2} (i.e. the KdV-like case) or in the form \eqref{60} (i.e. the NLS like case). }
 	 \begin{equation}
 	 \label{e:10}
 	 \cj \cl f=\la f.
 	 \end{equation}
We need to introduce a a real Hilbert space, so that $f\in X$, its dual $X^*$, so that $\cl:X\to X^*$, so that  the bilinear form $(u,v)\to \dpr{\cl u}{v}$ is a bounded symmetric bilinear form on $X\times X$. Next, we shall need to assume that $\cj$ has a domain $D(\cj)\subset X^*$, so that $\cj: D(\cj)\to X$, $\cj^*=-\cj$. Furthermore, ssume that there is an $\cl$ invariant decomposition of the base space in the form 
$$
X=X_-\oplus Ker[\cl]+\oplus X_+
$$
where (see Section 2.1, \cite{LZ}),  $\cl|_{X_-}<0$, $n(\cl):= dim(X_-)<\infty$,   $dim(Ker[\cl])<\infty$ and $\cl|_{X_+}\geq \de$, for some $\de>0$. In general, we will denote by $n(M)$ the (finite) number of negative eigenvalues (counted with multiplicities) of a generic self-adjoint operator $M$. 

Next, consider the finite dimensional generalized eigenspace at the zero eigenvalue, defined as follows 
$$
E_0=gKer[\cj \cl]=span[\cup_{k=1}^\infty [Ker[\cj\cl]^k] ]
$$
Note that $Ker[\cl]\subset E_0$ and introduce $\tilde{E}_0: E_0=Ker[\cl]\oplus \tilde{E}_0$. Consider the integer $k_0^{\leq 0}(\cl):= n(\cl|_{\tilde{E}_0})$.  Equivalently, taking an arbitrary basis in $\tilde{E}_0$, $\{\psi_1, \ldots, \psi_N\}\subset D(\cl)$, define $k_0^{\leq 0}(\cl)$ to be the number of negative eigenvalues of the $N\times N$ matrix $\cd=(\dpr{\cl \psi_i}{\psi_j})_{{i,j}, 1\leq i,j\leq N}$. 

Under these general assumptions,  it is proved in \cite{LZ} (see Theorem 2.3 and also Theorem 1, \cite{KKS2} for the case where $\cj$ has a bounded inverse) that 
\begin{equation}
\label{se:10} 
k_r+2 k_c+ 2k_0^{\leq 0}\leq n(\cl)-n(\cd), 
\end{equation}
where $k_r$ is the number of real and positive solutions $\la$ in \eqref{e:10} (i.e. real instabilities), $2 k_c$ is the number of solutions $\la$ in \eqref{e:10} with positive real part (i.e. modulational instabilities). 
 
\subsubsection{NLS-like problem}
For the eigenvalue problem in the form \eqref{60}, we have that $\cj$ is invertible and anti-symmetric, $\cj^{-1}=\cj^*=-\cj$ and  $X=H^2(\rd), X^*=H^{-2}(\rd), d\geq 1$. In addition,assume  that $\cj: Ker[\cl]\to (Ker[\cl])^\perp$. 
We now  introduce the matrix $\cd$ as follows. 

Let $Ker[\cl]=\{\phi_1, \ldots, \phi_n\}$, 
 then $\psi_j: \cj \cl \psi_j = \phi_j$. Note that the last equation has solution, since $ \cj^{-1} \phi_i \in Ker[\cl]^\perp$ and hence $\cl^{-1}[\cj^{-1} \phi_i]$ is well-defined. Hence the matrix $\cd$ is 
 \begin{equation}
 \label{dij}
 \cd_{i j}=\dpr{\cl \psi_i}{\psi_j}=\dpr{\cl^{-1} [\cj^{-1} \phi_i] }{\cj^{-1} \phi_j}=\dpr{\cl^{-1}
 	 [\cj \phi_i] }{\cj \phi_j}. 
 \end{equation}
By the index counting inequality \eqref{se:10}  if $n(\cl)\leq n(D)$, we can conclude that  spectral stability holds true, since the right-hand side of \eqref{se:10} is non-positive, hence all the indices on the left are zero as well.  

Next, we discuss $gKer[\cj \cl]$. We have at least $d+1$  elements in $Ker[\cl]$, namely $\phi_0:=\left(\begin{array}{c}
0 \\ \phi
\end{array}\right)$ and $\phi_j:=\left(\begin{array}{c}
\p_j \phi \\ 0
\end{array}\right), j=1, \ldots, d$. Assuming that $\phi\perp Ker[\cl_+]$ and $\nabla \phi\perp Ker[\cl_-]$, we can identify at least $d+1$  more elements of the generalized kernel $E_0$, namely $\psi_0=\left(\begin{array}{c}
\cl_+^{-1} \phi \\ 0
\end{array}\right)$ and 
$\psi_j=\left(\begin{array}{c}
0 \\ -\cl_-^{-1} \p_j \phi
\end{array}\right), j=1, \ldots, d$.   
This means that the algebraic multiplicity of the zero eigenvalue is at least $2(d+1)$, consisting of $d+1$ eigenfunctions and $d+1$ generalized eigenfunctions. One may wonder whether there is any more non-trivial elements in $gKer[\cj \cl]$. The non-degeneracy 
condition $\dpr{\cl_+^{-1} \phi}{\phi}\neq 0$, which appears in the statement of the main result is necessary condition that the Jordan block associated to the eigenvector  $\phi_0$ is exactly two dimensional. 
To this end,  assume that there is a third element, $q: \cj \cl q=\psi_0$. This would mean, that there is $q: \cl_- q=\cl_+^{-1} \phi$. By the self-adjointness of $\cl_-$, the solvability condition is exactly $\dpr{\cl_+^{-1} \phi}{\phi}\neq 0$. Indeed, $R(\cl_-)=Ker(\cl_-)^{\perp}=span \{\phi\}^\perp$,  so a third element in the Jordan cell for $\phi_0$ does not exist exactly when 
$\dpr{\cl_+^{-1} \phi}{\phi}\neq 0$.
\subsubsection{Kawahara-like problem}  For eigenvalues problem in the form \eqref{lin2} 
\begin{equation}
\label{e:103}
\p_x \cl f=\la f,
\end{equation}
where we set up  again $X=H^2(\rone), X^*=H^{-2}(\rone)$, while $, \cl=\cl_+, \cj=\p_x, \cj^*=-\cj$. This  satisfies  the  requirements of the theory put forward in the beginning of this section. Next, regarding the generalized kernel of $\p_x \cl_+$, we clearly have that $\phi'\in Ker[\cl]\subseteq Ker[\p_x \cl]$. Furthermore, if $\phi\perp Ker[\cl_+]$, there is additional element in $gKer[\p_x \cl]$, namely $\cl_+^{-1} \phi$, since $(\p_x \cl_+)^2[\phi']=\p_x \cl_+[\p_x \cl_+[\cl_+^{-1} \phi]]=0$. This means that the zero is multiplicity two eigenvalue for $\p_x \cl_+$, which is generated by the translational invariance.  
\subsubsection{Sufficient condition for spectral stability} Based on the inequality \eqref{se:10}, it is clear that spectral stability holds, if $n(\cl)=1$  and $n(\cd)\geq 1$. Furthermore, in both cases under considerations, and under the assumption $\phi\perp Ker[\cl_+]$, we have the vector $\psi=\cl_+^{-1} \phi$ in the generalized kernel of $\cj \cl$. Thus, $\cd_{11}=\dpr{\cl_+^{-1} \phi}{\phi}$, whence since $\cd_{11}<0$, we can assert that the matrix $\cd$ has at least one negative eigenvalue (since $\dpr{\cd e_1}{e_1}=\cd_{11}<0$, which would then imply stability. Thus, when we specify to the specific problems that we face, we can formulate  the following sufficient condition for spectral stability.  
\begin{corollary}
	\label{stability} 
	For the spectral problems \eqref{lin2} and \eqref{60}, spectral stability follows, provided
	\begin{itemize}
		\item $n(\cl_+)=1$, $\cl_-\geq 0$. 
		\item $\phi\perp Ker[\cl_+]$, $\dpr{\cl_+^{-1} \phi}{\phi}<0$. 
	\end{itemize}
\end{corollary}

\subsection{Necessary conditions for existence of \eqref{20}} 
We have the following Pohozaev identities. 

\begin{lemma}(Pohozaev's identities)
  Let  some smooth and decaying $\phi$ satisfy 
 \begin{equation}
  \Delta^2\phi+ \eps \sum_{j,k}^{n}b_jb_k\p_{j,k}\phi +\om\phi-|\phi|^{p-1}\phi =0 \label{EqPokhozaev}.
 \end{equation}
  Then
 \begin{eqnarray}
 \label{app:10}
  \int_{\rd}|\Delta \phi|^2dx &=& \frac{d(p-1) - 2 (p+1)}{2(p+1)} \int_{\rd}|\phi|^{p+1}dx  +\om \int_{\rd }|\phi|^2dx,\\
  \label{app:20}
  \eps   \int_{\rd}|\vec{b}\cdot \nabla \phi|^2dx &=&   \frac{d(p-1) - 4(p+1)}{2(p+1)} \int_{\rd}|\phi|^{p+1}dx  +2\om \int_{\rd }|\phi|^2dx, \\
    \label{app:30}
  (d(p-1) - 4(p+1)) \|\Delta \phi\|^2 &-& \eps(d(p-1) - 2 (p+1))\|\vec{b}\cdot \nabla \phi\|^2 + \om d(p-1)  \|\phi\|^2dx=0
 \end{eqnarray}
 \end{lemma}
 
 \begin{proof}

 Multiplying  \eqref{EqPokhozaev} by $ \phi  $ and integrating over $ \rd $ we get 
 \begin{equation*}
  \int_{\rd}|\Delta \phi|^2dx-\eps \int_{\rd}|\vec{b}\cdot \nabla \phi|^2dx-\int_{\rd}|\phi|^{p+1}dx+\om\int_{\rd }|\phi|^2dx=0. 
 \end{equation*}
 Also, multiplying \eqref{EqPokhozaev} by $ x\cdot\nabla\phi $ and integrating over $ \rd $ we get 
\begin{equation*}
 \left ( 2-\frac{d}{2}\right )\int_{\rd}|\Delta \phi|^2dx-\left (1-\frac{d}{2}\right ) \eps \int_{\rd}|\vec{b}\cdot \nabla \phi|^2dx+\frac{d}{p+1}\int_{\rd}|\phi|^{p+1}dx - \om\frac{d}{2}\int_{\rd}|\phi|^2dx = 0.
\end{equation*}

 Let $ A= \int_{\rd}|\Delta \phi|^2dx$, $ B=\eps \int_{\rd}|\vec{b}\cdot \nabla \phi|^2dx $, $ C=\int_{\rd}|\phi|^{p+1}dx $ and $ D= \int_{\rd}|\phi|^2dx$.

  Solving for $ A $ and $ B $ in terms of $ C $ and $ D $ we get
  \begin{equation*}
  \begin{cases}
  A&=\frac{d(p-1) - 2 (p+1)}{2(p+1)}C+\om D,\\
  B&=   \frac{d(p-1) - 4(p+1)}{2(p+1)}C+2\om D.\\
  \end{cases} 
    \end{equation*}
    which is \eqref{app:10} and \eqref{app:20}. The formula \eqref{app:30} follows similarly. 
  \end{proof}
\begin{corollary}
\label{cor:90}
If $d=1,2$, then $\om>0$. If $\eps=-1$ and $\om>0$, then $p<p_{\max}$. 

If $\vec{b}=0$,  then $\om>0$ and  $p<p_{\max}$.
\end{corollary}
\begin{proof}
If $d=1,2$, the first term on the right of \eqref{app:10} is negative, forcing the positivity of the second term, so  $\om>0$. Next, from the relation \eqref{app:20}, we see that if $\om>0, \eps=-1$, then $\frac{d(p-1) - 4(p+1)}{2(p+1)}<0$, or $p<p_{\max}$. 

If $\vec{b}=0$,  it is clear from \eqref{app:20} that either $\om>0$ and $p<p_{\max}$ or $\om<0$ and $p>p_{\max}$ (the second one being  impossible immediately for $d=1,2,3,4$). For $d\geq 5$, assume for a moment that $\om<0$ and $p>p_{\max}=\f{d+4}{d-4}$. Let us look at \eqref{app:10}. The second term is now negative, while for the first term, since $p>p_{\max}>\f{d+2}{d-2}$, we also conclude its negativity. It follows that the right hand side of \eqref{app:10} is negative a contradiction. Thus, $\om>0$, $p<p_{\max}$. 
\end{proof}
As we see from the results of Corollary \ref{cor:90}, the Pohozaev's identities are by themselves not strong enough to derive necessary conditions on $\om, p$ that are close to the sufficient ones. 

We believe that indeed, the necessary conditions are close to the ones required by \cite{L} to construct solutions of the constrained minimization problem \eqref{Lev}. Namely, we expect $p<p_{\max}$ and $\om>\f{b^2}{4}$ for $b>0$   to be necessary for existence of localized and smooth solutions to \eqref{EqPokhozaev} and \eqref{NLS22}. Let us show that in fact, these follow from a natural assumption on the spectrum for the operator $\cl_+$, namely that zero cannot be an embedded eigenvalue in the continuous spectrum of $\cl_+$. Let us note that while for second order Schr\"odinger operators $\ch=-\De+V$, this is generally the case\footnote{That is point spectrum does not embed into the continuous one}  under decay conditions on $V$, this is not the case for their fourth order counterparts, \cite{Gold}. In physically relevant situations however (and the case of $\cl_+$ certainly merits this designation), embedded eigenvalues should not  exist. If this is the case for $\cl_+$, we see that since by Weyl's theorem 
$$
\si_{a.c.}[\cl_+]=\si_{a.c.}(\De^2+ b \De+\om - p |\phi|^{p-1})=\si_{a.c.}(\De^2+ b \De +\om)=\left\{
\begin{array}{cc}
\om - \f{b^2}{4}  & b\geq 0 \\
\om & b<0
\end{array}
\right..
$$
Clearly, if zero is not embedded, it must be that $\om$ satisfies 
$
\om\geq \left\{
\begin{array}{cc}
\f{b^2}{4}  & b\geq 0 \\
0 & b<0
\end{array}
\right..
$
If that holds, at least in the case $b<0$, it follows from Corollary \ref{cor:90} 
that $p<p_{\max}$ as well.

 \section{Variational construction in the one dimensional case}
 \label{sec:3}
 We start with some preparatory results. 
 \subsection{Variational problem: preliminary steps}
 \label{sec:3.1}
 We now discuss the variational problem \eqref{70}. 
It is certainly not {\it a priori} clear that for a given $\la>0$, such a value is finite (that is $m_b(\la)>-\infty$) and non-trivial (i.e. $m_b(\la)<0$).  In fact, in some cases, it is not finite, as we show below.  Note that 
\begin{equation}
\label{73}
 \frac{m_b(\la)}{\lambda} =\inf_{\left\lVert \phi \right\rVert_2^2=1}\left \{\frac{1}{2}\int_{\rone}|\phi''|^2-b|\phi'|^2dx-\frac{\la^{\frac{p-1}{2}}}{p+1}\int_{\rone }|\phi|^{p+1}dx\right \}=\inf_{\left\lVert \phi \right\rVert_2^2=1}J[\phi].
 \end{equation}
This is, clearly, a non-increasing function. In particular, $ \frac{m_b(\la)}{\lambda}$ is differentiable a.e. and so is $m_b(\la)$. 
 Our considerations naturally split in two case, $b>0$ and $b<0$. 
 \subsubsection{The case $b<0$} 
 In this section, we develop criteria (based on the parameters in the problem), which address the question for finiteness and non-triviality of $m_b(\la)$. The next lemma shows this  for $p\in (1,5)$ and in addition, it establishes  that  $m_b(\la)=-\infty$ for $p>9$. 
 \begin{lemma} 
 \label{ponefive}
 For  $ p\in (1,5), b<0 $,  $ -\infty<m_b(\la)<0 $ for all $ \lambda>0 $.  For  $p \geq 9 $ then $ m_b(\la)=-\infty $ for all $ \lambda>0 $.

\end{lemma}
 \begin{proof}
 Let $ \phi_\ve(x)=\ve^{1/2}\phi(\ve x) $, where $ \left\lVert \phi \right\rVert_2^2=\lambda $. We have that 
\begin{equation} \label{iphiepsilon}
I[\phi_\ve]=  \frac{\|\phi''\|_{L^2}^2}{2}\ve^4 -\frac{b \|\phi'\|_{L^2}^2}{2}  \ve^2- 
\frac{\|\phi\|_{L^{p+1}}^{p+1}}{p+1}\ve^{\frac{p-1}{2}}. 
\end{equation}
 Since $ 0<\frac{p-1}{2}<2 $ for $ 1<p<5 $, we see that $ m_b(\la)<0 $ in this case by choosing $ \ve $ small enough. On the other hand, if $p>9$, it is clear that $\lim_{\ve\to \infty} I[\phi_\ve]=-\infty$, whence $m_b(\la)=-\infty$ in this case. 
 
By the GNS inequality 
 \begin{equation}
 \label{f:91} 
  \|\phi\|_{L^{p+1}(\rone)}\leq C_p \|\phi\|_{\dot{H}^{\f{1}{2}-\f{1}{p+1}}}\leq C_p \|\phi\|_{L^2}^{\f{3}{4}+\f{1}{2(p+1)}}  \|\phi''\|_{L^2}^{\f{1}{4}-\f{1}{2(p+1)}}, 
 \end{equation}
 we have  
\begin{eqnarray*}
I[\phi] &=& \frac{1}{2}\int_{\rone}|\phi''|^2-b|\phi'|^2dx-\frac{1}{p+1}\int_{\rone}|\phi|^{p+1}dx \\ 
&\geq & \frac{1}{2}\int_{\rone}|\phi''|^2-b|\phi'|^2dx -c_p\|\phi''\|_{L^2}^{\frac{p-1}{4}}\|\phi \|_{L^2}^{p+1-\frac{p-1}{4}} \\
&\geq & \frac{1}{4} \| \phi''\|_{L^2}^2- c_{p,\lambda,b}(\|\phi'' \|_{L^2}^{\frac{p-1}{4}}+1) \geq -\gamma,
\end{eqnarray*} 
for some $ \gamma>0$ because the function $ g(x)=\frac{1}{2}x^2-c_{p,\lambda}x^{\frac{p-1}{4}} $, clearly, has a negative minimum on $ [0,\infty) $ for $ p\in(1,9) $. Therefore, $ m_b(\la) \geq-\gamma>-\infty $ for $ p\in (1,9) $.
Letting $ \ve\to\infty $  in \eqref{iphiepsilon} shows  that $ m_b(\la)=-\infty  $ for $ p>9 $.

Consider now the case $p=9$. Clearly, for large $\la$, $m_b(\la)<0$, as it is evident from the formula \eqref{73}. Assuming that $m_b(\la)\in (-\infty, 0)$ for some $\la$, let $\phi$ be such that $m_b(\la)\leq I[\phi]< \f{m_b(\la)}{2}$.  Using $\phi_N$ as in the formula \eqref{iphiepsilon}, we see that $\|\phi_N\|_{L^2}^2=\la$, while for $N\geq 1$, we have 
$$
I[\phi_N]= N^4[\frac{\|\phi''\|_{L^2}^2}{2} -\frac{b \|\phi'\|_{L^2}^2}{2 N^2}   - 
\frac{\|\phi\|_{L^{10}}^{10}}{10}]\leq N^4[\frac{\|\phi''\|_{L^2}^2}{2} -\frac{b \|\phi'\|_{L^2}^2}{2}   - 
\frac{\|\phi\|_{L^{10}}^{10}}{10}] \leq N^4  \f{m_b(\la)}{2}
$$
But then 
$$
m_b(\la)\leq \liminf_N I[\phi_N]=-\infty, 
$$
a contradiction.

 \end{proof}
 Our next lemma shows that for $p\in [5,9)$, there is a threshold value $\la_p>0$, below which $m_b(\la)$ is trivial.  
 \begin{lemma}
 \label{le:2}
If $b<0$ and $p\in [5,9) $,  then there exists a finite number $ \la_p>0$ such that
\begin{itemize}
\item  for all $ \la\leq\la_p $ we have  $ m_b(\la)=0 $,
\item  for all $ \la>\la_p $  we have $ -\infty<m_b(\la)<0 $. 
\end{itemize} 
\end{lemma}
 \begin{proof}
Take $ \phi_\ve $ as in Lemma \ref{ponefive} with $ \left\lVert \phi \right\rVert_2^2=1 $. We have 
\begin{equation}\label{JleqZero}
 \frac{m_b(\la)}{\lambda}\leq\lim_{\ve\to 0}J[\phi_\ve]= 0.
\end{equation}
which implies that $ m_b(\la) \leq 0$.  Now, we are going to show that for each $ p\in [5,9] $ there exists a constant $ c_p>0 $ such that 
\begin{equation}
\label{90}
\inf_{\phi\neq0}\frac{\left\lVert\phi \right\rVert_2^{p-1}\left(\int_{\rone}|\phi''|^2-b|\phi'|^2dx\right)}{\int_{\rone}|\phi|^{p+1}dx}\geq c_p. 
\end{equation}
Using the GNS inequality \eqref{71}, we  get the following estimates for the $ L^{p+1} $ norm: 

\begin{align}
\left \lVert \phi \right\rVert_{p+1}^{p+1}
& \leq a_p \left\lVert \phi'' \right\rVert_2^{\frac{p-1}{4}}\left\lVert \phi \right\rVert_2^{\frac{3p+5}{4}}\nonumber\\
&\label{pplusoneest1}\leq a_p \left(\int_{\rone}|\phi''|^2-b|\phi'|^2dx\right)^{\frac{p-1}{8}}\left\lVert \phi \right\rVert_2^{\frac{3p+5}{4}},  
 \end{align}
 and
\begin{align}
\left \lVert \phi \right\rVert_{p+1}^{p+1}
& \leq b_p \left\lVert \phi' \right\rVert_2^{\frac{p-1}{2}}\left\lVert \phi \right\rVert_2^{\frac{3p+5}{4}}\nonumber\\
&\label{pplusoneest2}\leq b_p \left(\int_{\rone}|\phi''|^2-b|\phi'|^2dx\right)^{\frac{p-1}{4}}\left\lVert \phi \right\rVert_2^{\frac{p+3}{2}}. 
\end{align}

Note that for $ p\in[5,9)$, we have that $\f{p-1}{8}< 1\leq \f{p-1}{4}$. Therefore, interpolating between estimates \eqref{pplusoneest1} and \eqref{pplusoneest2} we get 
$$
\|\phi\|_{L^{p+1}}^{p+1} \leq  c_p \|\phi \|_{L^2}^{p-1} \int_{\rone}|\phi''|^2-b|\phi'|^2dx.
 $$
Thus we have that for all $ \phi\in H^2 $ with $ \left\lVert \phi \right\rVert_2^2=1 $ 
$$
 \int_{\rone}|\phi''|^2-b|\phi'|^2dx -\f{1}{c_p} \int_{\rone}|\phi|^{p+1}dx\geq0, 
 $$
this implies that for  $\la:  0<\lambda\leq \gamma_p=\left(\f{p+1}{c_p}\right)^{\frac{2}{p-1}} $, $ J[\phi]\geq0 $, which together with \eqref{JleqZero} implies that $ m_b(\la)=0 $. 

Observe that for a very large $\la$, the quantity 
$$
\inf_{\left\lVert \phi \right\rVert_2^2=1}\left \{\frac{1}{2}\int_{\rone}|\phi''|^2-b|\phi'|^2dx-\frac{\la^{\frac{p-1}{2}}}{p+1}\int_{\rone }|\phi|^{p+1}dx\right \}
$$
is strictly negative\footnote{which can be seen by fixing $\phi$ in the infimum and taking $\la>\la(\phi)$}, so $\la_p<\infty$. Clearly,  $ \lambda_p=\sup\{\gamma>0:\, m_b(\la)=0 \mbox{ for all } \lambda\leq \gamma\} $. 

\end{proof}

 \begin{lemma}\label{subslemma1}
Suppose $ b<0 $, $ 1<p<9 $ and $ -\infty<m_b(\la)<0 $. Let $ \phi_k $ be a minimizing sequence. Then, there exists a subsequence $ \phi_k $ such that:
$$
 \int_{\rone} |\phi_k''(x)|^2 dx\to L_1,\  \int_{\rone}|\phi_k'(x)|^2dx \to L_2, \ \int_{\rone}|\phi_k(x)|^{p+1}dx\to L_3,  
$$
where $ L_1>0 $, $ L_2> 0 $ and $ L_3> 0 $.
\end{lemma}
 \begin{proof}
 We have already established in Lemma \ref{ponefive}  that 
 \begin{equation}
 \label{300}
 I[\phi]\geq \frac{1}{4} \| \phi''\|_{L^2}^2- c_{p,\lambda,b}(\|\phi'' \|_{L^2}^{\frac{p-1}{4}}+1). 
 \end{equation}
Since, $\phi_k$ is minimizing, it follows that the sequence $ \{\int_{\rone}|\phi_k''(x)|^2 dx\}_k $ is bounded.  By GNS inequality, the sequences 
$ \{ \int_{\rone}|\phi_k'(x)|^2 dx\}_k $ and $  \int_{\rone}|\phi_k(x)|^{p+1}dx\}_k $ are bounded as well. Passing to a subsequence a couple of times we get a subsequence $\{\phi_k\} $ such that all of the above sequences converge. We claim that $ L_3 $ cannot be zero.  Indeed, otherwise, 
$$
m_b(\la)=\lim_k [\frac{1}{2}\int_{\rone} |\phi_k''(x)|^2 dx - \frac{b}{2}\int_{\rone}|\phi_k'(x)|^2dx]\geq 0
$$
which is a contradiction with the fact that $m_b(\la)<0$.   
By Sobolev embedding, neither $L_1$ nor $L_2$ could be zero, as this would force $L_3=0$, which we have shown to be impossible.

\end{proof}
 
 \subsubsection{ The case $b >0$. }
 \begin{lemma}
 	\label{le:k0} 
If $ b>0 $  and $ 1<p<9 $, then $ -\infty<m_b(\la)<0 $ for all $ \la>0 $.
\end{lemma}
\begin{proof}
Since $ 0<\frac{p-1}{2} < 4$, the dominant term in \eqref{iphiepsilon} is 
$\max(\ve^2, \ve^{\f{p-1}{2}})$, so if we just take $ \ve  $ small enough, we see that $ m_b(\la)<0 $. Boundedness from below follows from \eqref{300}. 
\end{proof}

\begin{lemma}
\label{le:GNS}
Let $p: 1<p<5$, $b>0$ and fix a constant $c$.  Then, the inequality 
\begin{equation}
\label{310}
 \norm{\phi}_{L^{p+1}}^{p+1}\leq c \norm{\phi}_{L^2}^{p-1}\left[\int_{\rone}|\phi''(x)|^2-b |\phi'(x)|^2+\frac{b^2}{4}|\phi(x)|^2dx\right]. 
 \end{equation}
 \underline{ cannot hold} for all $\phi\in H^2(\rone)$. 
 
 For $p\in [5,9]$, $b>0$, there is a $c_{b,p}$, so that 
 \begin{equation}
\label{311}
 \norm{\phi}_{L^{p+1}}^{p+1}\leq c \norm{\phi}_{L^2}^{p-1}\left[\int_{\rone}|\phi''(x)|^2-b |\phi'(x)|^2+\frac{b^2}{4}|\phi(x)|^2dx\right]. 
 \end{equation}

\end{lemma}
\begin{proof}
Let $p\in [5,9]$. Write 
$$
\int_{\rone}|\phi''(x)|^2-b |\phi'(x)|^2+\frac{b^2}{4}|\phi(x)|^2dx=\int_{\rone} |\hat{\phi}(\xi)|^2\left((2\pi \xi)^2 - \f{b}{2}\right)^2 d\xi. 
$$
Introducing $g$, so that $\hat{\phi}(\xi):=\hat{g}(2\pi \xi - \sqrt{\f{b}{2}}).$ Clearly, \eqref{311} is equivalent to the estimate 
\begin{equation}
\label{gl}
\norm{g}_{L^{p+1}}^{p+1}\leq c \norm{g}_{L^2}^{p-1} \int_{\rone} \hat{g}(\xi)|^2 |\xi|^2 |\xi-C_b|^2  d\xi 
\end{equation}
for some $C_b\neq 0$. We show \eqref{gl} as follows:  we decompose  the function in three regions - near  the two singularities $\xi=0$, $\xi=C_b$ and away from them. That is, for values of $|\xi|<<1$, we estimate  by Sobolev embedding and H\"older's inequality 
\begin{eqnarray*}
\norm{g_{<<1}}_{L^{p+1}} &\lesssim &  \|g_{<<1}\|_{\dot{H}^{\f{1}{2}-\f{1}{p+1}}} =  c \left(\int_{|\xi|<<1} |\hat{g}(\xi)|^2 |\xi|^{1-\f{2}{p+1}} d\xi\right)^{1/2} 
\lesssim \\
 &\lesssim & \|g\|_{L^2}^{\f{p-1}{p+1}} \left(\int_{|\xi|<<1} |\hat{g}(\xi)|^2 |\xi|^{\f{p-1}{2}}  d\xi \right)^{\f{1}{p+1}}   \lesssim  \norm{g}_{L^2}^{p-1} \int_{\rone} \hat{g}(\xi)|^2 |\xi|^2 |\xi-C_b|^2  d\xi. 
\end{eqnarray*}
Clearly, this last estimate holds  as long as $2\leq \f{p-1}{2}$ (since then $|\xi-C_b|\sim 1$, when $|\xi|<<1$), which is the same as $p\geq 5$. 
The estimate is similar, with the same constraint $p\geq 5$, at the singularity $\xi=C_b$. 

Finally, away from the two singularities, we have $|\xi|^2 |\xi-C_b|^2\sim |\xi|^4$, which means that following the estimates above, we need $\f{p-1}{2}\leq 4$,  which gives the other restriction $p\leq 9$.

Let now $p\in (1,5)$. Take a Schwartz function $\chi$ and then $\phi(x)=\chi(\eps x)$. Testing \eqref{310} for this choice of $\phi$ leads us to 
$\eps^{-1}\leq C \eps^{-\f{p-1}{2}}(\eps^3+\eps)$. This is a contradiction as $\eps\to 0+$, so \eqref{310} cannot hold. 
\end{proof}

\begin{lemma}\label{subslemma2}
Suppose $ b>0, \la>0 $ and  $ 1<p<9$.  Let $ \phi_k $ be a minimizing sequence for $\inf_{\|\phi\|_{L^2}^2=\la} I[\phi]$.  Then, assuming that 
\begin{itemize}
\item $p\in (1,5)$, $\la>0$, 
\item $p\in [5,9)$ and for some sufficiently large $\la_{b,p}$, $\la>\la_{b,p}$.  
\end{itemize} 
 Then, there exists a subsequence  $ \phi_{n_k} $,  such that:
$$
 \frac{1}{2}\int_{\rone}|\phi_{n_k}''(x)|^2 \to L_1,  \int_{\rone}|\phi_{n_k}'(x)|^2 \to L_2  \textup{ and } \int_{\rone}|\phi_{n_k}|^{p+1}dx\to L_3,  
$$
where $ L_1>0 $, $ L_2 > 0 $ and $ L_3 > 0 $.
\end{lemma}
\begin{proof}
First, by \eqref{300},  the quantity $\int_{\rone}|\phi_{k}''(x)|^2 dx$ is bounded. By Sobolev embedding so are the other two. 
 By passing to a subsequence (denoted again $\phi_k$), we can assume that they converge to three non-negative reals, $L_1, L_2, L_3$. 
 
 Suppose first that $L_3=0$. Then,   consider  the following minimization problem 
$$
\inf_{\norm {\phi}_2^2=\la}\frac{1}{2}\int_{\rone }|\phi''(x)|^2-b|\phi'(x)|^2dx := \inf_{\norm {\phi}_2^2=\la} \tilde{I}[\phi]. 
$$
Observe that since $
\tilde{I}[\phi]\geq I[\phi], 
$
we have
$$
\lim_k \tilde{I}[\phi_k]=\lim_k I[\phi_k]=\inf_{\norm {\phi}_2^2=\la} I[\phi]\leq \inf_{\norm {\phi}_2^2=\la} \tilde{I}[\phi]. 
$$
Thus,  $\phi_k$ is minimizing for $\tilde{I}$ as well and 
$$
\inf_{\norm {\phi}_2^2=\la} I[\phi]=\inf_{\norm {\phi}_2^2=\la} \tilde{I}[\phi]. 
$$
On the other hand, 
$
\inf_{\norm {\phi}_2^2=\la} \tilde{I}[\phi] 
$
  is easily seen to be $ -\frac{\la  b^2}{8} $. Indeed, for function $\phi: \|\phi\|_{L^2}^2=\la$, we have by Plancherel's 
  \begin{equation}
  \label{320}
 2 \tilde{I}[\phi] +\f{b^2}{4}\la= \int_{\rone }|\phi''(x)|^2-b|\phi'(x)|^2+\f{b^2}{4} \phi^2(x)dx=\int_{\rone} |\hat{\phi}(\xi)|^2\left|(2\pi \xi)^2  - \f{b}{2}\right|^2 d\xi\geq 0.
  \end{equation}
 whence $\inf_{\norm {\phi}_2^2=\la} \tilde{I}[\phi] \geq -\frac{\la  b^2}{8}.$ On the other hand, for any Schwartz function $\chi$, consider 
 $$
 \hat{\phi}_\epsilon(\xi):=\f{\sqrt{\la}}{\sqrt{\epsilon}\|\chi\|_{L^2}} \chi\left(\f{ \xi-\f{1}{2\pi}\sqrt{\f{b}{2}}}{\epsilon}\right)
 $$
  which has $\|\phi\|_{L^2}^2=\la$ and saturates the inequality \eqref{320} in the sense that 
  $$
\lim_{\epsilon\to 0+}  \int_{\rone} |\hat{\phi}_\epsilon(\xi)|^2\left|(2\pi \xi)^2  - \f{b}{2}\right|^2 d\xi \to 0. 
  $$
   Thus, $\inf_{\norm {\phi}_2^2=\la} I[\phi]=-\frac{\la  b^2}{8}$. 
  So , we have 
$$
-\frac{\la b^2}{8}=m_{b}(\la)\leq \frac{1}{2}\int_{\rone}|\phi''(x)|^2-b |\phi'|^2dx -\frac{1}{p+1}\int_{\rone}|\phi(x)|^{p+1}dx. 
$$
holds for all  $ \phi $ with $ \norm{\phi}_2^2=\la $. Applying this to an arbitrary $f$ and $\phi:= \sqrt{\la} \f{f}{\|f\|_{L^2}}$,  so that $\|\phi\|_{L^2}^2=\la$   the following inequality holds 
$$
 \frac{\la^{\frac{p-1}{2}}b^\frac{p-9}{4}}{p+1}\int_{\rone}|f(x)|^{p+1}dx\leq \frac{1}{2}\norm{f}_2^{p-1}\left(\int_{\rone}|f''(x)|^2- b |f'(x)|^2+\frac{b^2}{4}
 |f(x)|^2dx\right)
 $$
 for all $f\neq 0$. This last inequality however contradicts Lemma \ref{le:GNS} - for every $\la>0$, if $p\in (1,5)$ and for all large enough 
 $\la$, if $p\in [5,9)$.  Thus $L_3\neq 0$. Clearly, by Sobolev embedding $L_1> 0$, $L_2>0$, otherwise $L_3$ must be zero, which previously lead to a contradiction.

\end{proof}

 \subsubsection{ Strict sub-additivity}

\begin{lemma}
\label{le:7}
Let $ 1<p<9 $ and $ \la > 0 $ Then  for all $ \alpha\in (0,\la) $  we have

 \begin{equation} \label{StricSubadditivity}
 m_b(\la)<m_b(\alpha)+m_b(\lambda-\alpha).
 \end{equation} 
\end{lemma}
 \begin{proof}
First, suppose that $ 1<p<5 $ and $ b<0 $. Then 
$$
m_b(\la)=\frac{\la}{\al}\inf_{\norm{\phi}_2^2=\al}\{\frac{1}{2}\int_{\rone}|\phi''(x)|^2-b|\phi'(x)|^2dx -
\frac{(\la/\al)^{\frac{p-1}{2}}}{p+1}\int_{\rone}|\phi(x)|^{p+1}dx\}<\frac{\la}{\al}m_b{(\al)}, 
$$
where the last strict inequality holds because there exists a  minimizing sequence for 
$m_b(\al)$, which  has the property $ \lim_k \norm{\phi_k}_{p+1} >0$. This means that the function 
$\la\to \f{m_b(\la)}{\la}$ is strictly decreasing.  
Assuming that $ \al \in [\f{\la}{2}, \la) $ (and otherwise we work with $ \la-\al $) we get 
$$
 m_b(\la)<\frac{\la}{\al}m_{b}(\al)=m_b(\al)+\frac{\la-\al}{\al}m_{b}(\al)\leq m_{b}(\al)+m_{b}(\la-\al),
 $$ 
 where we have used $\f{m_b(\al)}{\al}\leq \f{m_b(\la-\al)}{\la-\al}$, since $\al\geq \la-\al$. This completes the case $p\in (1,5), b<0$. 
 
Let $ 5\leq p < 9 $ and $ b<0 $. Note that in this case, $m_b(x)$ is zero for small $x$, by Lemma \ref{le:2}. 
 So, there are three  possibilities:

\begin{enumerate}
 \item $ m_b(\al) =m_b(\la-\al)=0$. In this case \eqref{StricSubadditivity} trivially holds, since by assumption $m_b(\la)<0$. 
\item $ m_b(\la)<0 $, but $ m_b(\la-\al)=0 $. In this case we have
$$
m_b(\la)<\frac{\la}{\al}m_b(\al)=m_b(\al)+(\frac{\la}{\al}-1)m_b(\al)< m_b(\al)+m_b(\la-\al). 
$$
\item When both $m_b(\al), m_b(\la-\al)$ are negative, the proof is the same as in the case $ 1<p<5 $ for $ b<0$.  
\end{enumerate} 

 Next, we consider  the cases when $ b>0 $. In this case for all $ 1<p<5 $ and all $ \la>0 $ we have that $ -\infty<m_b(\la)<0 $. The proof is the same as in the case $b<0, p\in (1,5)$, since we never develop the complication that $m_b(\la)=0$ for any $\la>0$.  The case $p\in [5,9)$ and $\la>\la_{b,p}$ is similar as well. 
 \end{proof}

 \subsection{Existence of the minimizer}
 \label{minimizer}
Now, suppose 
$$
\left\{
\begin{array}{cc}
1<p<5 & \la> 0 \\
5\leq p<9 & \la>\la_{b,p} 
\end{array}
\right.
$$
  so that Lemma \ref{subslemma1} and Lemma \ref{subslemma2} hold. Let $ \{\phi_k\}_{k=1}^{\infty}\subset H^2  $ be a minimizing sequence, i.e.  
$$
  \int_{\rone}|\phi_k|^2 dx=\lambda, \qquad I[\phi_k]\to m_b(\la). 
  $$
Therefore, by passing to a further subsequence, by Lemma   \ref{subslemma1} and Lemma \ref{subslemma2}, we have 
$$
 \left\lVert \phi_k'' \right\rVert_2^2\to L_1>0,\quad  \left\lVert \phi_k' \right\rVert_2^2\to L_2>0,\quad \left\lVert \phi_k \right\rVert^{p+1}_{L^{p+1}}\to L_3>0.   
 $$  
Let $ \rho_k=|\phi_k|^2 $, so $\int \rho_k(x) dx=\la$. 
By the concentration compactness lemma of P.L.Lions (see Lemma 1.1, \cite{lions}), there is a subsequence (denoted again by $\rho_k$), so that  at least one of the following is satisfied: 
\begin{enumerate}
\item \textit{Tightness.} There exists $ y_k \in \rone $ such that for any $ \ve >0 $ there exists $ R(\ve) $ such that for all $ k $ 
$$
\int_{B(y_k,R(\ve))}\rho_k dx\geq \int_{\rone}\rho_k-\ve. 
$$
\item \textit{Vanishing.} For every $ R>0 $
$$
\lim\limits_{k\to \infty } \sup_{y\in \rone}\int_{B(y,R)}\rho_kdx=0. 
$$
\item \textit{Dichotomy.} There exists $ \alpha \in (0,\lambda) $, such that for any $ \ve>0 $ there exist $ R, R_k\to \infty,y_k $ and $ k_0 $ such that 
\begin{equation}
\label{350}
\left| \int_{B(y_k,R)}\rho_k dx-\alpha \right| < \ve, \ \left| \int_{R<|x-y_k|<R_k}\rho_k dx\right| < \ve, \ \ \left| \int_{R_k<|x-y_k|}\rho_k dx  -(\la-\al)\right| < \ve. 
\end{equation}
\end{enumerate}
 We proceed to rule out the dichotomy and vanishing  alternatives, which will leave us with tightness. 
\subsubsection{Dichotomy is not an option} 
 Assuming dichotomy, we have by \eqref{350} and $\int \rho_k(x) dx=\la$ that $ \left| \int_{R_k<|x-y_k|}\rho_k dx-(\la-\al) \right| < 2 \ve$. 
 Let $ \psi_1,\psi_2\in C^{\infty }(\rone) $, satisfying $ 0\leq\psi_1,\psi_2\leq 1 $  and 

$$
\psi_1(x) =\begin{cases}
1,&|x|\leq 1, \\
0,&|x|\geq2,   \\
\end{cases}
,\quad
\psi_2(x) =\begin{cases}
1,&|x|\geq 1, \\
0,&|x|\leq1/2,   \\
\end{cases}. 
$$ 
 Define $ \phi_{k,1} $ and $ \phi_{k,2} $ as follows:
 $$
  \phi_{k,1}(x)=\phi_k(x)\psi_1\left(\frac{x-y_k}{R_k/5}\right),\quad  \phi_{k,2}(x)=\phi_k(x)\psi_2\left(\frac{x-y_k}{R_k}\right).
  $$
  Clearly, for $ k $ large enough we have 
$$
\left| \int_{\rone}  \phi_{k,1}^2(x) dx-\alpha \right| < 2\ve \mbox{ and } \left| \int_{\rone} \phi^2_{k,2}(x)  dx  -(\lambda-\alpha) \right| < 2\ve. 
$$
  In fact, by taking a sequence $\ve_k\to 0$, we can find subsequence of $\phi_{k,1}, \phi_{k,2}$ (denoted again the same) and sequences  
  $ \{y_k\} _{k=1}^{\infty}\subset \rone$, $ \{R_k\}_{k=1}^{\infty} $ with $ R_k\ra \infty $ as $ k\ra \infty $, such that
\begin{equation}
\label{CCompactenessSeq}
\lim_{k\rightarrow\infty}\int_{\rone}|\phi_{k,1}|^{2} dx=\alpha,\quad  \lim_{k\rightarrow\infty}\int_{\rone}\left| \phi_{k,2} \right|^2 dx=\lambda-\alpha \mbox{ and } \int_{R_k/5<|x-y_k|<R_k}|\phi_k|^2 dx<\frac{1}{k}. 
\end{equation}
Consider $I[\phi_k]-I[\phi_{k,1}]-I[\phi_{k,2}]$. Using  \eqref{CCompactenessSeq} we get
\begin{align*}
&I[\phi_k]-I[\phi_{k,1}]-I[\phi_{k,2}]
=\frac{1}{2}\int_{\rone}|\phi_k''|^2-b|\phi_k'|^2dx-\frac{1}{p+1}\int_{\rone}|\phi_k|^{p+1}\\
&-\frac{1}{2}\int_{\rone}\left |\left (\phi_k\psi_1\left (\frac{x-y_k}{R_k/5}\right )\right )''\right |^2- 
b\left |\left(\phi_k\psi_1\left(\frac{x-y_k}{R_k/5}\right )\right)'\right|^2dx
+\frac{1}{p+1}\int_{\rone}\left |\left(\phi_k\psi_1\left(\frac{x-y_k}{R_k/5}\right)\right) \right|^{p+1}\\
&-\frac{1}{2}\int_{\rone}\left |\left (\phi_k\psi_2\left(\frac{x-y_k}{R_k}\right )\right)''\right |^2-
b\left |\left (\phi_k\psi_2\left(\frac{x-y_k}{R_k}\right)\right)'\right|^2dx
+\frac{1}{p+1}\int_{\rone}\left|\left(\phi_k\psi_2\left(\frac{x-y_k}{R_k}\right)\right) \right|^{p+1}\\
&=\frac{1}{2}\int_{\rone}\left(1-\psi_1^2\left(\frac{x-y_k}{R_k/5}\right)-\psi_2^2\left (\frac{x-y_k}{R_k}\right)\right) \left[|\phi_k''(x)|^2 - \f{b}{2}|\phi_k'(x)|^2\right] dx +\\
&+ \f{1}{p+1} \int_{\rone} |\phi_k(x)|^{p+1}\left(\psi_1^{p+1}\left(\frac{x-y_k}{R_k/5}\right)+\psi_2^{p+1}\left (\frac{x-y_k}{R_k}\right)-1\right)dx  +  E_k. 
\end{align*}
The error term $E_k$, contains only terms having at least one derivative on the cutoff functions, therefore generating $R_k^{-1}$. At the same time, there is at most one derivative falling on the $\phi_k$. So, we can estimate these terms away as follows 
$$
|E_k| \leq \f{C}{R_k} \int_{R_k/5<|x|<2 R_k} (|\phi_k(x)|^2+ |\phi'_k(x)|^2) dx\leq \f{C}{R_k} \|\phi_k\|_{L^2}(\|\phi_k\|_{L^2}+\|\phi_k''\|_{L^2}).
$$
 Since $ \sup_k \|\phi_k\|_{L^2}, \sup_k \|\phi_k''\|_{L^2}<\infty$, we conclude that $\lim_k E_k=0$.  For the next term, we have the positivity relation  
 $ \int_{\rone}\left(1-\psi_1^2\left(\frac{x-y_k}{R_k/5}\right)-\psi_2^2\left (\frac{x-y_k}{R_k}\right)\right) |\phi_k''(x)|^2  dx >0$.  
 Integration by parts yields 
 \begin{eqnarray*}
& &   \int_{\rone}\left(1-\psi_1^2\left(\frac{x-y_k}{R_k/5}\right)-\psi_2^2\left (\frac{x-y_k}{R_k}\right)\right) |\phi_k'(x)|^2  dx = \\
&=& -\int_{\rone} \phi_k(x) \f{d}{dx}[ \left(1-\psi_1^2\left(\frac{x-y_k}{R_k/5}\right)-\psi_2^2\left (\frac{x-y_k}{R_k}\right)\right)    \phi_k'(x)] dx 
\end{eqnarray*}
Thus, by H\"older's inequality 
 \begin{eqnarray*}
& &   |\int_{\rone}\left(1-\psi_1^2\left(\frac{x-y_k}{R_k/5}\right)-\psi_2^2\left (\frac{x-y_k}{R_k}\right)\right) |\phi_k'(x)|^2  dx| \leq  \\
&\leq & C \|\phi_k''\|_{L^2}  \|\phi_k\|_{L^2(R_k/5<|\cdot|<R_k)}+\f{C}{R_k} \|\phi_k'\|_{L^2} \|\phi_k\|_{L^2}.
\end{eqnarray*}
Note that since $R_k\to \infty$ and on the other hand $\|\phi_k\|_{H^2}$ is uniformly bounded in $k$, this term goes to zero, by the last estimate in \eqref{CCompactenessSeq}. 
  Finally, 
   \begin{eqnarray*}
& &    |\int_{\rone} |\phi_k(x)|^{p+1}\left(\psi_1^{p+1}\left(\frac{x-y_k}{R_k/5}\right)+\psi_2^{p+1}\left (\frac{x-y_k}{R_k}\right)-1\right)dx|
\leq \int_{ R_k/5<|x-y_k|<R_k} |\phi_k(x)|^{p+1} dx. 
 \end{eqnarray*}
Since by GNS 
$$
 \int_{ R_k/5<|x-y_k|<R_k} |\phi_k(x)|^{p+1} dx \leq C \|\phi_k''\|_{L^2}^{\f{p-1}{4}} \|\phi_k\|_{L^2(R_k/5<|\cdot|<R_k)}^{\f{3p+5}{4}}, 
$$
 and $\|\phi_k''\|_{L^2}$ is uniformly bounded in $k$,    we conclude that this term also goes to zero as $k\to \infty$. 

It follows that 
\begin{equation}
\label{400}
\liminf_{k\to \infty} \left[I[\phi_k]-I[\phi_{k,1}]-I[\phi_{k,2}]\right]\geq 0. 
\end{equation}
 Now, let $ \{a_k\}_{k=1}^{\infty} $ and $ \{b_k\}_{k=1}^{\infty} $ be sequences  such that 
$
\left\lVert a_k \phi_{k,1} \right\rVert_2^2=\alpha,\quad \left\lVert b_k \phi_{k,2} \right\rVert^2_2=\lambda-\alpha.   
$
Note that $ a_k,b_k\ra 1 $.  Using \eqref{400}, there is $\be_k: \lim_k \be_k=0$,  so that 
\begin{align*}
I[\phi_k]
&\geq  I[\phi_{k,1}]+I[\phi_{k,2}]+\beta_k\\
&\geq I[a_k \phi_{k,1}]+I[b_k \phi_{k,2}]+\beta_k-C(|1-a_k|+|1-b_k|)\\
&\geq m_b(\al)+m_b(\la-\al) +\beta_k-C(|1-a_k|+|1-b_k|). 
\end{align*}
where we have used that $\sup_k \|\phi_k\|_{H^2}<\infty$, the estimate $| I(\phi)-I(a \phi)|\leq C(\|\phi\|_{H^2}) |1-a|$ (which is a direct consequence of the definition of the functional $I[\cdot]$) and the definition of $m_b(z)$. 
Taking limits in $k$, we see that 
$$
m_b(\la)= \lim_k I[\phi_k]\geq m_b(\al)+m_b(\la-\al),
 $$
  which  is a contradiction with  the sub-additiivity of $m_b(\cdot)$ established in Lemma \ref{le:7}.   So, dichotomy cannot occur.

 \subsubsection{Vanishing does not occur} 
 \label{sec:3.2.2}
 Suppose vanishing occurs and $\ve>0$.  Let $ \phi \in C^\infty$ be such that 
$$
 \eta(x)=
\begin{cases}
1,&|x|\leq 1,\\
0, &|x|\geq 2.   \\
\end{cases} 
$$
 Using GNS we have for all $ R $ and $y\in \rone  $
\begin{eqnarray*}
\|\phi_k \|_{L^{p+1}\left(B(y,R)\right)}^{p+1}  &\leq &   \int_{B(y,R)}|\phi_k|^{p+1}dx \leq \int_{\rone }\left |\phi_k\eta\left (\frac{x-y}{R}\right)\right |^{p+1}dx \\
&\leq&  \left\lVert \left (\phi_k\eta\left (\frac{x-y}{R}\right)\right )'' \right\rVert_{L^2(\rone)}^{\frac{p-1}{4}}\left\lVert \phi_k \right\rVert_{L^2(B(y,2R))}^{\frac{3p+5}{4}} \leq C_{\eta,R} \left\lVert \phi_k \right\rVert_{L^2(B(y,2R))}^{\frac{3p+5}{4}}.
\end{eqnarray*}
 We can cover $ \rone $ with balls of radius $ 2 $ such that every point is contained in at most $ 3 $ balls, let it be  $ \{B(y_j,2)\} $. Moreover, we can choose these balls so that $ \{B(y_j,1)\} $ still covers $ \rone $. Choose $ N\in \N $ so large that for all $k>N$, 
 $$
 \int_{B(y,2)}|\phi_k|^2dx<\ve,
 $$
 for all $ y\in \rone $.  
 We can estimate the $ L^{p+1}(\rone ) $ norm of $ \phi_k$ as follows
\begin{align*}
\left\lVert \phi_k \right\rVert_{L^{p+1}(\rone)}^{p+1}\leq\sum_{j=1}^{\infty}\int_{B(y_j,1)}|\phi_{k}|^{p+1}dx\leq \sum_{j=1}^{\infty}C_{\eta,R}\left\lVert \phi_k \right\rVert_{L^2(B(y_j,2))}^{2}\left\lVert \phi_k \right\rVert_{L^2(B(y_j,2))}^{\frac{3p-3}{4}}\leq 3 C_{\eta,R} \ve^{\frac{3p-3}{4}}\left\lVert \phi_k \right\rVert_{L^{2}(\rone)}^2.
\end{align*}
So,  we get that $ \left\lVert \phi_k \right\rVert_{L^{p+1}(\rone)}^{p+1}\ra 0 $ as $ k\ra \infty $ which is a contradiction. Therefore, the sequence $ \rho_k=|\phi_k|^2  $ is tight.
 \subsubsection{Existence of the minimizer} 
 \label{sec:3.2.3} 
 We have that there exists a sequence $ \{y_k\}_{k=1}^{\infty} $ such that for all $ \ve>0 $ there exists $ R(\ve) $ such that 
$$
\int_{|x|>R(\ve)}|\phi_k(y_k+x)|^2dx<\ve. 
$$
Define $ u_k(x):=\phi_k(y_k+x) $. The sequence $ \{u_k\}_{k=1}^{\infty}\subset H^2 $  is bounded, therefore there exists a weakly convergent subsequence( renamed to $ \{u_k\}_{k=1}^{\infty} $), say,  to $ u \in H^2$ . By the tightness and the compactness criterion on $ L^2(\rn) $,  the sequence $ \{u_k\}_{k=1}^{\infty} $  has a strongly convergent subsequence in $ L^2(\rone) $, say, to $ \widetilde{u} \in H^2$. Since weak convergence on $ H^2 $ implies weak convergence on $ L^2 $, we have that $ u=\widetilde{u} $ by uniqueness of weak limits. In addition, $\|u\|_{L^2}^2=\lim_k \|u_k\|_{L^2}^2=\la$, so $u$ satisfies the constraint. 

We also have that $ u_k $ converges to $ u $ in $ L^{p+1} $ norm. Indeed, using GNS inequality we get 
\begin{align*} \left\lVert u_k-u \right\rVert_{L^{p+1}(\rone)}
&\leq \left\lVert (u_k-u)'' \right\rVert_{L^{2}(\rone)}^{\frac{p-1}{4(p+1)}}\left\lVert u_k-u \right\rVert_{L^{2}(\rone)}^{1-\frac{p-1}{4(p+1)}}\\
&\leq C\left\lVert u_k-u\right\rVert_{L^{2}(\rone)}^{1-\frac{p-1}{4(p+1)}}\ra 0 \mbox{ as }k\ra \infty.
\end{align*}
Also, since 
$$
\|u_k'-u'\|_{L^2}^2\leq \|u_k''-u''\|_{L^2}\|u_k-u\|_{L^2}\leq (\|u_k''\|_{L^2}+\|u''\|_{L^2}) \|u_k-u\|_{L^2}, 
$$
we conclude that $\lim_k \|u_k'-u'\|_{L^2}=0$, and in addition  $\lim_k \int (u_k'(x))^2 dx \to \int (u'(x))^2 dx$. 

Finally,  by the lower semicontinuity of the $L^2$ norm with respect to weak convergence, we have 
$\liminf_k \int_{\rone}|u_k''|^2\geq \int_{\rone}|u''|^2$.  We conclude that 
\begin{equation*}
 \liminf_k \frac{1}{2}\int_{\rone}|u_k''|^2-b|u_k'|^2dx-\frac{1}{p+1}\int_{\rone }|u_k|^{p+1}dx\geq \frac{1}{2}\int_{\rone}|u''|^2-b|u'|^2dx-\frac{1}{p+1}\int_{\rone }|u|^{p+1}dx,  
 \end{equation*}
whence we have that $m_b(\la)\geq I[u] $, therefore $I(u)=m_b(\la)$ and $ u $ is a minimizer.
 \label{sec:2.3}

 \subsection{Euler-Lagrange equation}
 \begin{proposition}
 \label{prop:14}
 Let $p\in (1,9), \la>0$,  be so that  
 \begin{itemize}
 \item $1<p<5, \la>0$
 \item $5\leq p<9,  \la>\la_{b,p}>0$. 
 \end{itemize}
 Then, there exists a function $\om(\la)>0$, so that  
 the minimizer of the constrained minimization problem \eqref{70} $\phi=\phi_\la$ constructed  in Section \ref{sec:2.3},   satisfies the Euler-Lagrange equation
 \begin{equation}
 \label{450}
 \phi_\la''''+b\phi_{\la}''-|\phi_\la|^{p-1}\phi_\la+\om(\la)\phi_{\la}=0
 \end{equation}
where 
$$
\om(\la) = \frac{1}{\la}\int_\rone b(\phi_\la')^2+|\phi_\la|^{p+1}-(\phi_\la'')^2dx. 
$$
In addition, $n(\cl_+)=1$, that is $\cl_+$ has exactly one negative eigenvalue.  In fact $\cl_+|_{\{\phi_\la\}^\perp}\geq 0$. 
 \end{proposition}
 \begin{proof}
 We have shown that minimizers for the constrained minimization problem exists in the two cases described above, for both $b>0$ and $b<0$. 
 
  Consider $ u_\de = \sqrt{\la}\frac{\phi_\la +\de h}{\norm{\phi_\la +\de h}}$, where $ h $ is a test function. Note that $\|u_\de\|_{L^2}^2=\la$, 
  so it satisfies the constraint. Expanding $ I[u_\de] $ in powers of $ \de $ we obtain 

\begin{align*}
I[u_\de] &= m_b(\la) + \\
&+ \de \left[ \int_\rone \phi_\la'' h'' -bh'\phi_\la'-h|\phi_\la|^{p-1}\phi_\la dx  +\frac{1}{\la}\int_\rone b(\phi_\la')^2+|\phi_\la|^{p+1}-(\phi_\la'')^2dx\int_\rone \phi_\la h dx \right]\\
& +\frac{\de^2}{2}\left[\int_\rone (h'')^2  -b (h')^2 - p h^2 \left| \phi_\la \right|^{p-1}dx\right] \\
&+ 
 \frac{\de^2}{\la} \dpr{h}{\phi}    \int_\rone (p+1)h|\phi|^{p-1}\phi  +2b h'\phi_\la' -2h''\phi_\la''dx \\
& +\frac{\de^2 }{2 \la^2} \dpr{h}{\phi}^2  \int_\rone (p+3)\left| \phi_\la \right|^{p+1} +4b(\phi_\la')^{2} -4(\phi_\la'')^{2}dx +\\
& +   \frac{\de^2}{2 \la} \|h\|^2  \int_\rone \left| \phi_\la \right|^{p+1}+b(\phi_\la')^2-(\phi_\la'')^2dx  +O(\de^3).
\end{align*}
Using only the first order in $\de$ information and the fact that $I[u_\de]\geq m_b(\la)$ for all $\de\in \rone$, we conclude that 
$$
\dpr{\phi_\la}{h''''}+b\dpr{\phi_\la}{h''} -\dpr{|\phi_\la|^{p-1}\phi_\la+\om(\la) \phi_\la}{h}=0
$$
where $\om(\la) = \frac{1}{\la}\int_\rone b(\phi_\la')^2+|\phi_\la|^{p+1}-(\phi_\la'')^2dx$. Since this is true for any test function $h$, we conclude that $\phi_\la$ is a distributional solution of  the Euler-Lagrange equation \eqref{450}. According to Proposition \ref{prop:nm}, this turns out to be a solution in stronger sense, in particular $\phi_\la\in H^4(\rone)$. 

Now, using the fact that the function $g_h(\de):=I[u_\de]$ has a minimum at zero, we also conclude that $g_h''(0)\geq 0$. This is of course valid for all $h$, but in order to simplify the expression, we only look at $h: \|h\|=1$, which are orthogonal to the wave $\phi_\la$, i.e. $\dpr{h}{\phi_\la}=0$. This implies that 
$$
\dpr{h''''+b h''+\om(\la) h-p |\phi_\la|^{p-1} h}{h}\geq 0.
$$
In other words, $\dpr{\cl_+ h}{h}\geq 0$, whenever $h: \|h\|=1, \dpr{h}{\phi_\la}=0$.  This  is exactly the claim that $\cl_+|_{\{\phi_\la\}^\perp}\geq 0$.  In particular, this implies that the second smallest eigenvalue of $\cl_+$ is non-negative or $n(\cl_+)\leq 1$. On the other hand, since $\dpr{\cl_+ \phi_\la}{\phi_\la}=-(p-1) \int |\phi_\la(x)|^{p+1} dx<0$, it follows that there is a negative eigenvalue or $n(\cl_+)=1$. 
 \end{proof}

 \section{Variational construction in higher dimensions} 
 \label{sec:4} 
 In this section, we follow the approach and constructions from Section \ref{sec:3}. Most, if not all, of the steps go through essentially unchanged, save for the numerology, 
 which is of course impacted by the dimension $d$. Thus, we will be just indicating the main points, without providing full details, where the arguments follow closely the one dimensional case.    
 
 Recall that we work with the variational problem \eqref{700}. Again, we introduce 
$$
m_b(\lambda)=\inf_{\phi\in H^2\cap L^{p+1} ,\left\lVert \phi \right\rVert_2^2=\lambda}I[\phi].
$$
Note  that since 
\begin{equation}
\label{733}
 \frac{m_b(\la)}{\lambda} =\inf_{\left\lVert \phi \right\rVert_2^2=1}
 \left \{  \frac{1}{2}\int_{\rd}[|\De \phi (x)|^2-\eps |\vec{b}|^2 |\p_{x_1}\phi(x)|^2]dx-\frac{\la^{\f{p-1}{2}}}{p+1}\int_{\rd}|\phi(x)|^{p+1}dx  \right \}, 
 \end{equation}
 the function $\la\to \f{m_b(\la)}{\la}$ is non-increasing, we conclude that  $m_b(\la)$ is differentiable a.e. 
 As we have previously discussed, the case $\eps=1$ seems much more technically complicated, and it is to be addressed in a subsequent publication \cite{KST}. 
 
 We concentrate on the case $\eps=-1$. 
We have the following regarding $m_{\vec{b}, \la}$. 
\begin{lemma}
\label{le:40}
Let $\eps=-1$. Then, 
\begin{itemize}
\item For $p\in (1, 1+\f{8}{d+1})$ and $\la>0$, we have that $-\infty<m_{\vec{b}}(\la)<0$, 
\item For $p\in (1, 1+\f{8}{d})$, $m_{\vec{b}}(\la)>-\infty$, 
\item 
For $p\geq 1+\f{8}{d}$, $m_{\vec{b}, \la}=-\infty$ for all $\la>0$. 
\end{itemize}
\end{lemma}
\begin{proof}
The proof goes through the same steps as in Lemma \ref{ponefive}. 
Pick $\phi_\de=\de^{\f{d+1}{2}} \phi(\de^2 x_1, \de x')$, with $\|\phi\|_{L^2}^2=\la$. Clearly, $\|\phi_\de\|_{L^2}^2=\la$, while 
$$
I[\phi_\de]=\f{\de^4 \|\De' \phi\|^2+\de^8 \|\p_{x_1 x_1} \phi\|_{L^2}^2}{2} + \f{ |\vec{b}|^2 \|\phi_{x_1}\|^2}{2} \de^4 - \f{\|\phi\|_{L^{p+1}}^{p+1}}{p+1} \de^{\f{(d+1)(p-1)}{2}}.
$$
Clearly, for $\de$ small enough and $p<1+\f{8}{d+1}$, the last term is dominant, so $m_b(\la)<0$. Similarly, using $ \psi_\de = \de^{\frac{d}{2}}\phi(\de x) $ we obtain 
$$
 I[\psi_\de ] = \frac{\de^4 \norm{\Delta \phi}^2+ \de^2  |\vec{b} |^2 \norm{\phi_{x_1}}^2}{2}- \frac{\norm{\phi}_{p+1}^{L^{p+1}}}{p+1}\de^{\frac{d(p-1)}{2}},  
$$
and taking the limit $\de\to \infty$ yields 
$m_b(\la)=-\infty$, for $p>1+\f{8}{d}$. 

Next, by GNS, we have that 
 $$
 \|\phi\|_{L^{p+1}(\rd)}\leq C_p \|\phi\|_{\dot{H}^{d(\f{1}{2}-\f{1}{p+1})}}\leq C_p 
 \|\phi\|_{L^2}^{1-d(\f{1}{4}-\f{1}{2(p+1)})}  \|\De \phi\|_{L^2}^{d(\f{1}{4}-\f{1}{2(p+1)})}.
 $$
Thus,  
\begin{eqnarray*}
I[\phi] &=&  \frac{1}{2}\int_{\rd}[|\De \phi (x)|^2+|\vec{b}|^2 |\p_{x_1}\phi(x)|^2]dx-\frac{1}{p+1}\int_{\rd}|\phi(x)|^{p+1}dx  \\ 
&\geq & \frac{1}{2}\int_{\rd}|\De \phi|^2 +|\vec{b}|^2 |\p_{x_1}\phi(x)|^2dx -c_p\|\De \phi\|_{L^2}^{d\f{p-1}{4}}\|\phi \|_{L^2}^{p+1-d \frac{p-1}{4}} \\
&\geq & \frac{1}{4} \| \De \phi \|_{L^2}^2- c_{p,\lambda,b}\|\De \phi \|_{L^2}^{d\frac{p-1}{4}} \geq -\gamma,
\end{eqnarray*} 
where in the last inequality, we have used that $p<1+\f{8}{d}$ (whence $d\f{p-1}{4}<2$) and hence $\| \De \phi \|_{L^2}^2$ is dominant. The fact that $m_b(\la)=-\infty$, when $p=1+\f{8}{d}$ follows in the same fashion as in Lemma \ref{ponefive}. 
\end{proof}
Next, we present a technical lemma. 
\begin{lemma}
\label{le:13} 
For $1+\f{8}{d+1}\leq p<1+\f{8}{d}$, there is $C_p$, so that for all functions $g$, 
\begin{equation}
\label{600}
\|g\|_{L^{p+1}(\rd)}^{p+1}\leq C_p \|g\|_{L^2}^{p-1} \int_{\rd} |\De g|^2+ |\p_{x_1} g|^2 dx
\end{equation}
For $p\in (1, 1+\f{8}{d+1})$, such an estimate cannot hold. 

\end{lemma}
\begin{proof}
We apply the  Sobolev embedding in the variables $x_1$ and then in $x'=(x_2, \ldots, x_d)$
\begin{equation}
\label{lop}
\|g\|_{L^{p+1}(\rd)} \lesssim  \||\nabla_{x'}|^{(d-1)(\f{1}{2}-\f{1}{p+1})} |\nabla_{x_1}|^{(\f{1}{2}-\f{1}{p+1})} g\|_{L^2(\rd)}.
\end{equation}
Next, by  Plancherel's, H\"older's  inequality  and Young's inequality 
\begin{eqnarray*}
& & \||\nabla_{x'}|^{(d-1)(\f{1}{2}-\f{1}{p+1})} |\nabla_{x_1}|^{(\f{1}{2}-\f{1}{p+1})} g\|_{L^2(\rd)} = \left(\int_{\rd} |\hat{g}(\xi)|^2 |\xi'|^{(d-1)(1-\f{2}{p+1})}|\xi_1|^{1-\f{2}{p+1}} d\xi\right)^{1/2} \\
&\lesssim & \|g\|_{L^2}^{\f{p-1}{p+1}} \left(\int_{\rd} |\hat{g}(\xi)|^2  |\xi'|^{(d-1)\f{p-1}{2}}|\xi_1|^{\f{p-1}{2}} d\xi \right)^{\f{1}{p+1}}\lesssim 
\|g\|_{L^2}^{\f{p-1}{p+1}} \left(\int_{\rd} |\hat{g}(\xi)|^2 [|\xi'|^4 + |\xi_1|^{\f{q'(p-1)}{2}}] d\xi \right)^{\f{1}{p+1}},
\end{eqnarray*} 
where $q=\f{8}{(d-1)(p-1)}$. Clearly, \eqref{600} follows, provided 
$
2\leq \f{q'(p-1)}{2}\leq 4.
$
Solving this inequality yields exactly $1+\f{8}{d+1}\leq p<1+\f{8}{d}$. 

If $p<1+\f{8}{d+1}$, take $\phi=\chi(\eps^2 x_1, \eps x')$ in \eqref{600}. Assuming the validity of \eqref{600}, we obtain a contradiction for $\eps<<1$. 
\end{proof}

The next two lemmas are  the generalizations  of Lemma \ref{le:2} and Lemma \ref{subslemma1} to higher dimensions. 
 \begin{lemma}
 \label{le:22}
If  $\eps=-1$ and $p\in [1+\f{8}{d+1},1+\f{8}{d}) $,  then there exists a finite number $ \la_{\vec{b},p}>0$ such that
\begin{itemize}
\item  for all $ \la\leq\la_{\vec{b},p} $ we have  $ m_b(\la)=0 $,
\item  for all $ \la>\la_p $  we have $ -\infty<m_b(\la)<0 $. 
\end{itemize} 
\end{lemma}
\begin{proof}
The inequality $m(\la)\leq 0$ follows in the same way as in Lemma \ref{le:2}. Then, by Lemma \ref{le:13}, we have  
\begin{equation}
\label{des}
\inf_{\phi\neq 0} \f{\|\phi\|_{L^2}^{p-1} \int_{\rd} [|\De \phi|^2-\eps |\vec{b}|^2 |\phi_{x_1}|^2] dx}{\int_{\rd} |\phi|^{p+1} dx} \geq c_{\vec{b}, p}>0.
\end{equation}
Thus, for all $\phi\in H^2(\rd)$, we have 
$$
 \int_{\rd} [|\De \phi|^2-\eps |\vec{b}|^2 |\phi_{x_1}|^2] dx -\f{c_{\vec{b}p}}{\la^{p-1}} \int_{\rd} |\phi|^{p+1} dx\geq 0, 
$$
which by \eqref{733} implies that  for $\la\leq \la_{\vec{b},p}:=\left(\f{c_{\vec{b},p}(p+1)}{2}\right)^{\f{2}{p-1}}$, $m_{\vec{b}}(\la)\geq 0$. Since we always have the opposite inequality, this implies $m_{\vec{b}}(\la)=0$, when $\la$ is small enough. 
Note that for very large $\la$, the quantity in \eqref{733} is clearly negative, so this implies that $\la_{\vec{b}, p}<\infty$. 
\end{proof}
 
 The next lemma is the generalization of Lemma \ref{subslemma1} to the higher dimensional case. Its proof follows an identical arguments and it is thus omitted. 
 \begin{lemma}\label{le:44}
Suppose $ \eps=-1$, $ p\in (1, 1+\f{8}{d}) $ and $ -\infty<m_b(\la)<0 $. That is 
\begin{itemize}
\item $p\in (1, 1+\f{8}{d+1}),  \la>0$
\item $p\in [1+\f{8}{d+1},1+\f{8}{d})$ and $\la>\la_{\vec{b}, p}$. 
\end{itemize}
 Let $ \phi_k $ be a minimizing sequence for the constrained minimization problem \eqref{700}. Then, there exists a subsequence $ \phi_k $ such that:
$$
 \int_{\rd} |\De \phi_k(x)|^2 dx\to L_1,  \int_{\rd}|\p_{x_1} \phi_k(x)|^2dx \to L_2, \  \int_{\rd}|\phi_k(x)|^{p+1}dx\to L_3,  
$$
where $ L_1>0 $, $ L_2 > 0 $ and $ L_3 > 0 $.
\end{lemma}
 
 \subsection{Existence of minimizers} 
 Before we go ahead with the  existence of minimizers, we need an analog of Lemma \ref{le:7}. 
 Their proofs in the higher dimensional case goes in an identical manner. 
 \begin{lemma}
 \label{707}
 Let $1<p<1+\f{8}{d}$ and $\la>0$. Then $\la\to m_{\vec{b}, p}(\la)$ is strictly subadditive. That is, for every $\al\in (0, \la)$, 
 $$
 m_{\vec{b}, p}(\la)<m_{\vec{b}, p}(\al)+m_{\vec{b}, p}(\la-\al)
 $$
 In addition, $\la\to m_{\vec{b}, p}(\la)$ is twice differentiable a.e. 
 \end{lemma}
 With the basic results in place, we can now proceed to establish the existence of the minimizers of \eqref{700}. Supposing 
 $$
\left\{
\begin{array}{cc}
1<p<1+\f{8}{d+1} & \la> 0 \\
1+\f{8}{d+1}\leq p<1+\f{8}{d} & \la>\la_{b,p} 
\end{array}
\right.
$$
 we take a minimizing sequence $\{\phi_k\}\subset H^2(\rd)$, with $I[\phi_k]\to m_{\vec{b}, p}(\la)$. 
 By eventually passing to a subsequence, we can without loss of generality assume, by using   Lemma \ref{le:44},  
 $$
 \frac{1}{2}\int_{\rd}|\De \phi_{n_k}(x)|^2 \to L_1,  \int_{\rd}|\p_{x_1} \phi_{n_k}(x)|^2 \to L_2  \textup{ and } 
 \int_{\rd}|\phi_{n_k}|^{p+1}dx\to L_3,  
$$
 where\footnote{For conciseness, we use $\phi_k$, instead of $\phi_{n_k}$}  $ L_1>0 $, $ L_2>0 $ and $ L_3>0 $. The next task is to show that this sequence does not split nor vanish.  
 The absence of splitting is established in the same way as  the first part of  Section \ref{minimizer}. 
 
 Next, we rule out vanishing. The proof presented in Section \ref{minimizer} works for $d=1,2,3,4$, but breaks down in $d\geq 5$, so let us present another one that works in all dimensions.  More concretely, for all $R>0$ and $y\in \rd$ and a cutoff function $\eta$ introduced in Section \ref{sec:3.2.2},  we have by the GNS inequality 
 \beqn
 \|\phi_k\|_{L^{p+1}(B(y,R))}^{p+1} &\leq &  \int_{\rd} |\phi_k(x) \eta\left(\f{|x-y|}{R}\right)|^{p+1} dx\lesssim 
 \|\phi_k \eta_R\|_{\dot{H}^{d\left(\f{1}{2}-\f{1}{p+1}\right)}}^{p+1} \lesssim \\
 &\lesssim &  \|\De[\phi_k \eta_R]\|_{L^2}^{(p+1)\f{d}{2}\left(\f{1}{2}-\f{1}{p+1}\right)} \|\phi_k \eta_R\|_{L^2}^{(p+1)-(p+1)\f{d}{2}\left(\f{1}{2}-\f{1}{p+1}\right)} 
 \eeqn
Since $p<1+\f{8}{d}$, it follows that $(p+1)\f{d}{2}\left(\f{1}{2}-\f{1}{p+1}\right)<2$. In addition $\|\phi_k \eta_R\|_{L^2}\leq \|\phi_k\|_{L^2(B(y,2R)}$, whence 
$$
 \|\phi_k\|_{L^{p+1}(B(y,R))}^{p+1}\leq C_{R, \eta}  \|\phi_k\|_{H^2(B(y, 2R))}^2 \|\phi_k\|_{L^2(B(y,2R))}^{p-1}.
$$
 So, if we assume that vanishing occurs, then for every $\ve>0$, we will be able to cover $\rd$ with balls of radius $1$,  say $B(y_j,1)$, 
 so that $\int_{B(y_j,3)} |\phi_k(x)|^2 dx<\ve$. Then, 
 \beqn
 \left\lVert \phi_k \right\rVert_{L^{p+1}(\rd)}^{p+1}  &\leq & \sum_{j=1}^{\infty}\int_{B(y_j,1)}|\phi_{k}|^{p+1}dx\leq 
 \sum_{j=1}^{\infty}C_{\eta,R}\left\lVert \phi_k \right\rVert_{H^2(B(y_j,2))}^{2}\left\lVert \phi_k \right\rVert_{L^2(B(y_j,2))}^{p-1}\leq \\
 &\leq& 
 10  C_{\eta,R} \ve^{\f{p-1}{2}}\left\lVert \phi_k \right\rVert_{H^{2}(\rd)}^2.
 \eeqn 
 Clearly, since $ \left\lVert \phi_k \right\rVert_{H^{2}(\rd)}$ is uniformly bounded in $k$, we conclude that $\|\phi_k\|_{L^{p+1}}\to 0$, which is in a 
  contradiction with $\lim_k \int_{\rd}|\phi_{k}|^{p+1}dx\to L_3>0$. 
  
  From here, it follows that the sequence $\rho_k=|\phi_k(x)|^2$ is tight and the existence of the minimizer is done as in Section \ref{sec:3.2.3}. 
 
 The Euler-Lagrange equation, together with the appropriate properties of the linearized operators is done similar to Proposition \ref{prop:14}. 
  \begin{proposition}
 \label{prop:144}
 Let $p\in (1,1+\f{8}{d}), \la>0$,  be so that  
 \begin{itemize}
 \item $1<p<1+\f{8}{d+1}, \la>0$
 \item $1+\f{8}{d+1} \leq p<1+\f{8}{d},  \la>\la_{b,p}>0$. 
 \end{itemize}
 Then, there exists a function $\om(\la)>0$, so that  
 the minimizer of the constrained minimization problem \eqref{700} $\phi=\phi_\la$ satisfies the Euler-Lagrange equation
 \begin{equation}
 \label{4501}
 \De^2 \phi_\la+ \eps |\vec{b}|^2 \p_{x_1}^2 \phi_{\la} -|\phi_\la|^{p-1}\phi_\la+\om(\la)\phi_{\la}=0
 \end{equation}
In addition, $n(\cl_+)=1$, that is $\cl_+$ has exactly one negative eigenvalue. Finally,   $\cl_-\geq 0$, with a simple  eigenfunction at zero, i.e. $Ker[\cl_-]=span[\phi_\la]$.  \end{proposition}
As we mentioned above, the proof goes along the lines of Proposition \ref{prop:14}. The only new element are the statements about $\cl_-$, which we now prove. 
  Note that by direct inspection, $\cl_-[\phi_\la]=0$, by \eqref{4501}, so zero is an eigenvalue. 
 Assuming that there is a negative eigenvalue, say $\cl_-[\psi]=-\si^2 \psi, \|\psi\|=1$, we clearly would have $\psi\perp \phi_\la$. In addition, since\footnote{This is an obvious statement, once we realize that $\phi_\la$ cannot vanish on an interval. Indeed, otherwise, since it solves the fourth order equation \eqref{4501}, it  follows  that $\phi_\la$ is trivial, which it is not.} $\cl_+<\cl_-$, 
 \begin{eqnarray*}
 & &  \dpr{\cl_+ \psi}{\psi}< \dpr{\cl_- \psi}{\psi}=-\si^2 \\
 & &  \dpr{\cl_+ \phi_\la}{\phi_\la}<0.
 \end{eqnarray*}
 This would force $n(\cl_+)\geq 2$, a contradiction. Thus, $\cl_-\geq 0$. Finally, $0$ is a simple eigenvalue of $\cl_-$ along the same line of reasoning. Indeed, take $\psi: \cl_-\psi=0, \psi\perp \phi_\la$. Again, we conclude $n(\cl_+)\geq 2$, which leads to a contradiction.

 \subsection{Discussion of the proof of Theorem \ref{theo:NLS2}: existence of the waves} 
 We do not provide an extensive review of the existence claims in Theorem \ref{theo:NLS2} ,as this would be repetitious, but we would like to make  a few notable points. We work with the variational problem \eqref{701}, where we set up $b=-1$ for simplicity as this will not affect the calculations.  
 
 Our goal in this section is to clarify the range of indices in $p$. More concretely, we have the following analogue of Lemmas \ref{le:13}.  
 \begin{lemma}
\label{le:133} 
For $1+\f{4}{d}\leq p<1+\f{8}{d}$, 
\begin{equation}
\label{612}
\|g\|_{L^{p+1}(\rd)}^{p+1}\leq C_p \|g\|_{L^2}^{p-1} \int_{\rd} |\De g|^2+ |\nabla  g|^2 dx
\end{equation}
For $p\in (1, 1+\f{4}{d})$, such an estimate cannot hold. 
\end{lemma}
The proof proceeds in a similar fashion, so we omit it. A combination of arguments in the flavor of the proofs for Lemma \ref{le:40} and Lemma \ref{le:22} leads us to the following variant of Lemma \ref{le:22} and Lemma \ref{le:44}.
 \begin{lemma}
 \label{le:224}
If  $b<0$ and $p\in [1+\f{4}{d},1+\f{8}{d}) $,  then there exists a finite number $ \la_{b,p}>0$ so that
\begin{itemize}
\item  for all $ \la\leq\la_{b,p} $ we have  $ m_b(\la)=0 $,
\item  for all $ \la>\la_p $  we have $ -\infty<m_b(\la)<0 $. 
\end{itemize} 
In addition, assuming that  $-\infty<m_{b}(\la)<0$, that is
\begin{itemize}
\item $p\in (1, 1+\f{4}{d}),  \la>0$
\item $p\in [1+\f{4}{d},1+\f{8}{d})$ and $\la>\la_{b, p}$. 
\end{itemize}
and $ \phi_k $ be a minimizing sequence for the constrained minimization problem \eqref{700},   there exists a subsequence $ \phi_k $ such that:
$$
 \int_{\rd} |\De \phi_k(x)|^2 dx\to L_1,  \int_{\rd}|\nabla \phi_k(x)|^2dx \to L_2, \  \int_{\rd}|\phi_k(x)|^{p+1}dx\to L_3,  
$$
where $ L_1>0 $, $ L_2 > 0 $ and $ L_3 > 0 $.
\end{lemma}
 With these tools at hand, the existence of the waves follows in the same manner as before, so we omit the details. 
 
 \section{Stability of the normalized waves}
 \label{sec:5}
 Interestingly, the proof of the spectral stability proceeds by a common argument, both for the Kawahara and the fourth order NLS case. By Proposition \ref{stability}, it suffices to show that $n(\cl_+)=1$, $\cl_-\geq 0$, $\phi_\la\perp Ker[\cl_+]$ and to verify that the index $\dpr{\cl_+^{-1} \phi_\la}{\phi_\la}<0$. Indeed, the condition $n(\cl_+)=1$ was already verified as part of the variational construction, see Proposition \ref{prop:14} and \ref{prop:144}. Similarly, $\cl_-\geq 0$ was verified in the higher dimensional case in Proposition \ref{prop:144}. 
 
 \subsection{Weak non-degeneracy and non-positivity of the Vakhitov-Kolokolov  quantity }
  \begin{lemma}
  	\label{le:wnon}
  	For each constrained minimizer $\phi_\la$, we have that $\phi_\la\perp Ker[\cl_+]$. 
  \end{lemma}
  \begin{proof}
  	Take any element of $Ker[\cl_+]$, say $\Psi: \|\Psi\|_{L^2}=1$. We need to show $\dpr{\Psi}{\phi_\la}=0$.  
  	To this end, consider $\Psi-\|\phi_\la\|^{-2} \dpr{\Psi}{\phi_\la} \phi_\la\perp \phi_\la$. Recall that due to the construction $\cl_+|_{\{\phi_\la\}^\perp}\geq 0$.  We have 
  	$$
  	0\leq 	\dpr{\cl_+[\Psi-\|\phi_\la\|^{-2} \dpr{\Psi}{\phi_\la} \phi_\la\la]}{\Psi-\|\phi_\la\|^{-2} \dpr{\Psi}{\phi_\la} \phi_\la}=\|\phi_\la\|^{-4}  \dpr{\Psi}{\phi_\la}^2 \dpr{\cl_+ \phi_\la}{\phi_\la}\leq 0,
  	$$
  	where we have used that $\dpr{\cl_+ \phi_\la}{\phi_\la}=-(p-1) \int |\phi_\la|^{p+1}<0$. The only way the last chains of inequalities is non-contradictory, is if $\dpr{\Psi}{\phi_\la}=0$, which is the claim. 
  \end{proof}
  Our next result  is a general lemma, which is  of independent interest. 
   \begin{lemma}
  	\label{le:93} 
  	Suppose that $\ch$ is a self-adjoint operator on a Hilbert space $X$, so that $\ch|_{\{\xi_0\}^\perp}\geq 0$. Next, assume $\xi_0\perp Ker[\ch]$, so that $\ch^{-1} \xi_0$ is well-defined. 
  	Finally, assume $\dpr{\ch \xi_0}{\xi_0}\leq 0$. Then
  	$$
  	\dpr{\ch^{-1} \xi_0}{\xi_0}\leq 0.
  	$$ 
  \end{lemma}
  \begin{proof}
  	We can without loss of generality assume that $\|\xi_0\|=1$.  
  	Consider 
  	$
  	\ch^{-1} \xi_0  - \dpr{\ch^{-1} \xi_0}{\xi_0} \xi_0\perp \xi_0. 
  	$
  	It follows that 
  	\begin{eqnarray*}
  		0 &\leq & \dpr{\ch[ \ch^{-1} \xi_0  - \dpr{\ch^{-1} \xi_0}{\xi_0} \xi_0]}{ \ch^{-1} \xi_0  - \dpr{\ch^{-1} \xi_0}{\xi_0} \xi_0}= \\
  		&=& \dpr{\xi_0- \dpr{\ch^{-1} \xi_0}{\xi_0}  \ch \xi_0}{ \ch^{-1} \xi_0  - \dpr{\ch^{-1} \xi_0}{\xi_0} \xi_0}=\\
  		&=& - \dpr{\ch^{-1} \xi_0}{\xi_0} \dpr{\ch \xi_0}{\ch^{-1} \xi_0} +\dpr{\ch^{-1} \xi_0}{\xi_0}^2 \dpr{\ch\xi_0}{\xi_0}=\\
  		&=& - \dpr{\ch^{-1} \xi_0}{\xi_0} +\dpr{\ch^{-1} \xi_0}{\xi_0}^2 \dpr{\ch\xi_0}{\xi_0} \leq - \dpr{\ch^{-1} \xi_0}{\xi_0},
  	\end{eqnarray*}
  	where we have used the assumption $\dpr{\ch\xi_0}{\xi_0} \leq 0$. 
  	It follows that 
  	$\dpr{\ch^{-1} \xi_0}{\xi_0}\leq 0$, which is the claim. 
  \end{proof}
  {\bf Remark:} Unfortunately, it is impossible to conclude that $\dpr{\ch^{-1} \xi_0}{\xi_0}<0$, based on the assumptions made in Lemma \ref{le:93}. It turns out that such a statement is in general false, that is it is in general impossible to rule out $\dpr{\ch^{-1} \xi_0}{\xi_0}\neq 0$. 
  
  To that end, consider  the following example\footnote{We owe this to a generous remark made by an anonymous referee in response to our initial claims to the contrary.}: Take $H=\rtwo$ and $\ch=\left(\begin{array}{cc}
-1 & 1 \\ 1 & 0
  \end{array}\right)$, $\xi_0=\left(\begin{array}{c}
  1 \\  0
  \end{array}\right)$, which has $Ker[\ch]=\{0\}$, $\dpr{\ch \xi_0}{\xi_0}=-1<0$, while $\dpr{\ch^{-1} \xi_0}{\xi_0}=0$. Nevertheless, we always have $\dpr{\ch^{-1} \xi_0}{\xi_0}\leq 0$ as claimed in Lemma \ref{le:93}. 
  
  \subsection{Conclusion of the proof of spectral stability}

  Apply Lemma \ref{le:93} to the vector $\xi_0:=\phi_\la$ and the operator $\ch:=\cl_+$. Recall that as a byproduct  of the construction of $\phi_\la$, we have established  the property $\cl_+|_{\{\phi_\la\}^\perp}\geq 0$. By Lemma \ref{le:wnon}, we have that $\phi_\la\perp Ker[\cl_+]$. Finally, $\dpr{\cl_+ \phi_\la}{\phi_\la}<0$ was established as well (and used repeatedly throughout). Thus, we conclude that $\dpr{\cl_+^{-1} \phi_\la}{\phi_\la}\leq 0$. Clearly, our additional assumption, namely $\dpr{\cl_+^{-1} \phi_\la}{\phi_\la}\neq 0$ guarantees that $\dpr{\cl_+^{-1} \phi_\la}{\phi_\la}<0$, which is enough for the spectral stability by Corollary \ref{stability}.  It would be interesting to see whether one can   prove  $\dpr{\cl_+^{-1} \phi_\la}{\phi_\la}\neq 0$ in a straightforward manner, instead of making it an extra requirement.
  
  These arguments establish rigorously the spectral stability of the waves for the Kawahara made in  Theorem \ref{theo:Kawstab}   and in the high dimensional  fourth order NLS problems  in 
  Theorem \ref{theo:NLS} and Theorem \ref{theo:NLS2}.



\begin{thebibliography}{99}

\bibitem{A} J.P. Albert, {\emph Positivity properties and stability of solitary-wave solutions of model equations for long
waves}, {\em  Comm. PDE}, {\bf 17} (1992), p. 1--22. 

\bibitem{ACN} T  de Andrade, F.  Crist\'ofani, F.  Natali, \emph{Orbital stability of periodic traveling wave solutions
for the Kawahara equation}, {\em J. Math. Phys. }  {\bf  58}, (2017), 051504. 

\bibitem{Pava} J.  Angulo Pava,  \emph{On the instability of solitary-wave solutions for fifth-order water wave models.} {\em  Electron. J. Differential Equations} 
 2003, No. {\bf 6}, 18 pp.
 
 \bibitem{Ben} T. B. Benjamin,  \emph{The stability of solitary waves.} {\em
 	 Proc. Roy. Soc. London Ser. A}, {\bf  328}, (1972), p. 153--183.
 
 \bibitem{BCSN} D. Bonheure,  J. B. Casteras,  E. dos Santos,  R. Nascimento,{\emph 
 	Orbitally stable standing waves of a mixed dispersion nonlinear Schr\"odinger equation}, 
 {\em SIAM J. Math. Anal.}, {\bf 50},  (2018), no. 5, 5027--5071.
 
 

\bibitem{BD} T.  Bridges, G. Derks,  \emph{Linear instability of solitary wave solutions of the Kawahara equation and its generalizations.} 
{\em SIAM J. Math. Anal.} {\bf  33}  (2002), no. 6, p. 1356--1378. 

\bibitem{CL} T. Cazenave, P. L. Lions,
\emph{Orbital Stability of Standing Waves
for Some Nonlinear Schr\"odinger Equations,} {\em  Comm. Math. Phys.} {\bf  85} (1982), p.  549--561. 

 
\bibitem{CG} W. Craig, M.D. Groves, {\emph Hamiltonian long-wave approximations to the water-wave problem}, {\em Wave Motion},  {\bf 19} 
 (1994), p. 367--389.

 \bibitem{FL} R. Frank;, E. Lenzmann, \emph{Uniqueness of non-linear ground states for fractional Laplacians in $\rone$}. {\em Acta Math.} 
 {\bf  210}  (2013), no. 2,  p. 261--318.
 
 \bibitem{Gold} M. Goldberg, personal communication. 
 
 \bibitem{G} M. D. Groves,  \emph{Solitary-wave solutions to a class of fifth-order model equations}, {\em  Nonlinearity}, {\bf 11} 
(1998), p. 341--353.
 
 \bibitem{IS} A.T. Ill'ichev, A.Y. Semenov, \emph{Stability of solitary waves in dispersive media described
by a fifth-order evolution equation,} {\em  Theor. Comput. Fluid Dyn.} {\bf  3}  (1992), p. 307--326.

\bibitem{FSS} W. Feng, M. Stanislavova, A. Stefanov, {\emph On the spectral stability of ground states of semilinear Schr\"odinger abd Klein-Gordon equations with fractional dispersion},   {\em Comm. Pure. Appl. Anal.}, {\bf 17}, (2018), no.4, p. 1371--1385.

 \bibitem{GSS} M. Grillakis, J. Shatah, W. Strauss, Stability theory of solitary wavs in the presence of symmetry I, J. Funct. Anal., 74(1987), p. 160-197.


\bibitem{HLS} M. Haragus; E. Lombardi; A. Scheel, \emph{Spectral stability of wave trains in the Kawahara equation.} {\em  J. Math. Fluid Mech.} 
{\bf  8}  (2006), no. 4,  p. 482--509.

\bibitem{HS} J. K. Hunter, J. Scheurle, \emph{Existence of perturbed solitary wave solutions to a model
equation for water waves}, {\em Phys. D} {\bf  32}  (1988),  p. 253--268.


 \bibitem{KKS} T. M. Kapitula, P. G. Kevrekidis, B. Sandstede, \emph{Counting eigenvalues
via Krein signature in infinite-dimensional Hamitonial systems,}  {\em Physica D}, {\bf 3-4}, (2004), p. 263--282. 

\bibitem{KKS2} T. Kapitula,P. G. Kevrekidis, B. Sandstede, \emph{ Addendum: "Counting eigenvalues via the Krein signature in
 infinite-dimensional Hamiltonian systems'' [Phys. D 195 (2004), no. 3-4,
 263--282] }
 {\em Phys. D}  {\bf  201}  (2005),  no. 1-2, 199--201.

  
\bibitem{KP} T. Kapitula, K. Promislow, Spectral and Dynamical Stability of
Nonlinear Waves, 185, Applied Mathematical Sciences, 2013.


 \bibitem{KS} T. Kapitula, A. Stefanov,  {\emph A Hamiltonian-Krein (instability) index theory for solitary waves to KdV-like eigenvalue problems.} {\em 
  Stud. Appl. Math.} {\bf  132},  (2014), no. 3, p. 183--211. 
  
\bibitem{K}   V.I. Karpman,  \emph{Stabilization of soliton instabilities by higher-order dispersion: KdV-type equations}, {\em Phys. Lett. A}, {\bf  210}, 
 (1996), p. 77--84. 
 
 
 \bibitem{K6} V.I. Karpman, \emph{Stabilization of soliton instabilities by higher-order dispersion: Fourth-order nonlinear
Schr\"odinger-type equations,} {\em  Phys. Rev. E}, {\bf 53} (1996), R1336.
 
 \bibitem{KS1}   V.I. Karpman,  A. Shagalov, {\emph  Solitons and their stability in high dispersive systems. I. Fourth order
nonlinear Schr\"odinger-type equations with power-law nonlinearities,} {\em  Phys. Lett. A}, {\bf 228} (1997),
pp. 59--65.
 
\bibitem{KS2}   V.I. Karpman,  A. Shagalov, {\emph Stability of solitons described by nonlinear Schr\"odinger-type equations with higher-order dispersion. } 
{\em Phys. D} {\bf  144} (2000), no. 1-2, p. 194--210. 

\bibitem{K1} V.I. Karpman, \emph{Lyapunov approach to the soliton stability in highly dispersive systems. I. Fourth order nonlinear Schr\"odinger equations.} 
{\em  Phys. Lett. A} {\bf 215} (1996), no. 5-6, p. 254--256. 

\bibitem{Kaw} R. Kawahara, \emph{Oscillatory solitary waves in dispersive media,} {\em  J. Phys. Soc. Japan,} {\bf 33}  (1972), pp. 260--264.

\bibitem{KST} P. G. Kevrekidis, A. Stefanov, Y. Tsolias, J. Maraver \emph{Quartic generalizations of the nonlinear Schr\"odinger model in two-dimensions:
	theoretical analysis and numerical computations}, in preparation. 

\bibitem{KO} S. Kichenassamy, P.J. Olver, \emph{Existence and nonexistence of solitary wave solutions to
higher-order model evolution equations}, {\em SIAM J. Math. Anal.}, {\bf  23},  (1992), p. 1141--1166.

\bibitem{L1}  S.  Levandosky,  \emph{A stability analysis for fifth-order water-wave models}, {\em  Phys. D}, {\bf 125} (1999), p. 222--240.
 

\bibitem{L} S. Levandosky,  \emph{Stability of solitary waves of a fifth-order water wave model.} {\em  Phys. D} {\bf  227} (2007), no. 2, 
p. 162--172.
 


 

\bibitem{LZ} Z. Lin, C.  Zeng, {\emph Instability, index theorem, and exponential trichotomy for Linear Hamiltonian PDEs}, available at 
https://arxiv.org/abs/1703.04016


\bibitem{lions} P.L. Lions, {\emph The concentration-compactness principle in the calculus of variations. The locally compact case, part I}, 
{\em  Annales de l'I. H. P.},  section C, tome 1, no 2 (1984), p. 109-145. 

\bibitem{NA} F. Natali, A. Pastor, {\emph The fourth order dispersive Nonlinear Schr\"odinger Equation: orbital stability of a standing wave}, {\em SIAM J. Applied Dyn. Sys.}, {\bf 14} (2015), no. 3, p.1326--1347.


 \bibitem{Pel} D. Pelinovsky, {\emph Spectral stability on nonlinear waves in KdV-type evolution equations},  Nonlinear physical systems, p. 377--400, 
Mech. Eng. Solid Mech. Ser., Wiley, Hoboken, NJ, 2014. 

 
 
\bibitem{W} Z. Wang,  {\emph Stability of Hasimoto solitons in energy space for a fourth order nonlinear Schr\"odinger type equation}, {\em Discrete Contin. Dyn. Syst.} 
{\bf  37}  (2017), no. 7,  p. 4091--4108. 
 











\end{thebibliography}
\end{document}